\documentclass[11pt,a4paper]{amsart}

\usepackage[utf8]{inputenc}
\usepackage[T1]{fontenc}

\usepackage{geometry}
\usepackage{setspace}
\usepackage{amsmath}
\usepackage{amssymb}
\usepackage{amsthm}
\usepackage{mathtools}
\usepackage{xcolor}
\usepackage{microtype}
\usepackage{needspace}

\usepackage{aliascnt}

\theoremstyle{plain}
\newtheorem{theorem}{Theorem}[section]

\newaliascnt{lemma}{theorem}
\newtheorem{lemma}[lemma]{Lemma}
\aliascntresetthe{lemma}

\newaliascnt{proposition}{theorem}
\newtheorem{proposition}[proposition]{Proposition}
\aliascntresetthe{proposition}

\newaliascnt{corollary}{theorem}
\newtheorem{corollary}[corollary]{Corollary}
\aliascntresetthe{corollary}

\theoremstyle{definition}
\newaliascnt{problem}{theorem}
\newtheorem{problem}[problem]{Problem}
\aliascntresetthe{problem}

\newaliascnt{definition}{theorem}
\newtheorem{definition}[definition]{Definition}
\aliascntresetthe{definition}

\theoremstyle{remark}
\newaliascnt{remark}{theorem}
\newtheorem{remark}[remark]{Remark}
\aliascntresetthe{remark}

\usepackage[colorlinks=true,
            linkcolor=purple,
            citecolor=teal,
            urlcolor=violet]{hyperref}

\usepackage[nameinlink,noabbrev]{cleveref}

\crefname{theorem}{Theorem}{Theorems}
\Crefname{theorem}{Theorem}{Theorems}
\crefname{lemma}{Lemma}{Lemmas}
\Crefname{lemma}{Lemma}{Lemmas}
\crefname{proposition}{Proposition}{Propositions}
\Crefname{proposition}{Proposition}{Propositions}
\crefname{corollary}{Corollary}{Corollaries}
\Crefname{corollary}{Corollary}{Corollaries}
\crefname{problem}{Problem}{Problems}
\Crefname{problem}{Problem}{Problems}
\crefname{definition}{Definition}{Definitions}
\Crefname{definition}{Definition}{Definitions}
\crefname{remark}{Remark}{Remarks}
\Crefname{remark}{Remark}{Remarks}
\crefname{section}{Section}{Sections}
\Crefname{section}{Section}{Sections}
\crefname{subsection}{Subsection}{Subsections}
\Crefname{subsection}{Subsection}{Subsections}
\crefname{equation}{equation}{equations}
\Crefname{equation}{Equation}{Equations}

\newcommand{\R}{\mathbb{R}}
\newcommand{\N}{\mathbb{N}}
\newcommand{\C}{\mathbb{C}}
\newcommand{\Sphere}{\mathbb{S}}
\newcommand{\dist}{\operatorname{dist}}

\newcommand{\cond}[2]{\mu\left(#1,#2\right)}
\newcommand{\condtotal}[1]{\mu\left(#1\right)}
\newcommand{\Energy}{E_{\log}}
\newcommand{\Energyminimal}{\mathcal{E}_{\log}}

\begin{document}

\title[Elliptic Fekete points]{Elliptic Fekete points:
separation, Riesz energies, and polynomial conditioning}

\author[U. Etayo]{Ujué Etayo}
\address{Departamento de Matemáticas, CUNEF Universidad, C/\,Almansa 101, 28040 Madrid, Spain}
\email{ujue.etayo@cunef.edu}

\thanks{The author was supported by the FBBVA starting grant associated with the Jos\'e Luis Rubio de Francia Prize; by the Ram\'on y Cajal Programme of the Spanish Ministry of Science, Innovation and Universities, through the Agencia
Estatal de Investigaci\'on (AEI), under grant RYC2024-049105-I, co-funded by the European Social Fund Plus (ESF+); and by grant PID2025-173363NA-I00, funded by MICIU/AEI/10.13039/501100011033.
Part of this work was carried out at Institut Mittag-Leffler
during the workshop ``Metric Algebraic Geometry''
(June 22--26, 2026).
The author gratefully acknowledges the institute's hospitality.}

\date{\today}

\begin{abstract}
Elliptic Fekete configurations minimize logarithmic energy on the sphere. They also provide a natural source of complex polynomials with controlled Shub--Smale condition number. The link between these two problems is the product of the distances from each point to all the other points: logarithmic energy controls the geometric mean of these products, whereas local conditioning depends on each product separately. This paper studies the passage from global information to uniform estimates at individual nodes, exploring its connections with quantitative equidistribution, Riesz energies, and polynomial conditioning.
We improve separation bounds for elliptic Fekete configurations and obtain quantitative Riesz-energy estimates for both these configurations and spherical roots of polynomials with optimal-order global conditioning.
For $0<s<2$, these estimates also improve explicit
upper bounds for the minimal Riesz $s$-energy.
\end{abstract}

\maketitle
\begingroup
\singlespacing
\endgroup
\section{Introduction}\label{sec:introduction}

We study several related problems concerning point configurations on the $2$-dimensional sphere and the conditioning of the associated polynomials.

\subsection{Notation and energies}\label{sec:notation}

Let $\Sphere^2\subset\R^3$ be the unit sphere, equipped with the normalized surface measure $\sigma$, so that $\sigma(\Sphere^2)=1$. We write
\[
\omega_d=\{\hat{z}_1,\ldots,\hat{z}_d\}\subset\Sphere^2
\]
for a configuration of $d$ distinct points, unless multiplicities are explicitly allowed. All distances in energy and product expressions are Euclidean chordal distances inherited from $\R^3$, and all energies count \emph{ordered} pairs.
The logarithmic and Riesz $s$-energies are defined, respectively, by
\begin{align*}
\Energy(\omega_d)
&:=-\sum_{\substack{1\le i,j\le d\\ i\ne j}}
\log\|\hat{z}_i-\hat{z}_j\|,
\\
E_s(\omega_d)
&:=\sum_{\substack{1\le i,j\le d\\ i\ne j}}
\|\hat{z}_i-\hat{z}_j\|^{-s},\qquad s\in\R\setminus\{0\}.
\end{align*}
For $s>0$, the local Riesz potential at a node is
\begin{equation*}
U_{s,i}(\omega_d):=\sum_{\substack{1\le j\le d\\j\ne i}}
\frac{1}{\|\hat{z}_i-\hat{z}_j\|^s},
\qquad
E_s(\omega_d)=\sum_{i=1}^d U_{s,i}(\omega_d).
\end{equation*}
The relevant continuous quantities, with these normalizations, are
\begin{equation}\label{def:kappa}
\kappa:=\int_{\Sphere^2}\log\frac{1}{\|p-q\|}\,\mathrm{d}\sigma(p)
=\frac12-\log2,
\qquad q\in\Sphere^2,
\end{equation}
and
\begin{equation}\label{eq:continuous-riesz-potential}
\int_{\Sphere^2}\frac{\mathrm{d}\sigma(p)}{\|p-q\|^s}
=\frac12\int_0^2 t^{1-s}\,\mathrm{d}t
=\frac{2^{1-s}}{2-s},\qquad 0<s<2.
\end{equation}
For general background on discrete energies, see \cite{BorodachovHardinSaff2019}.

For any finite configuration, let us denote by
\begin{equation}\label{eq:products-definition}
H_{\omega_d}(p):=\prod_{j=1}^d\|p-\hat{z}_j\|,
\qquad
\mathcal Q(\omega_d):=\max_{p\in\Sphere^2}H_{\omega_d}(p),
\end{equation}
and define the nodal products and squared norms by
\begin{align*}
P_i&:=H_{\omega_d\setminus\{\hat{z}_i\}}(\hat{z}_i)
=\prod_{j\ne i}\|\hat{z}_i-\hat{z}_j\|,
\\
I_d&:=\int_{\Sphere^2}H_{\omega_d}(p)^2\,\mathrm{d}\sigma(p),
\qquad
I_{d,i}:=\int_{\Sphere^2}H_{\omega_d\setminus\{\hat{z}_i\}}(p)^2\,\mathrm{d}\sigma(p).
\end{align*}
The dependence of $P_i$, $I_d$, and $I_{d,i}$ on the configuration $\omega_d$ is suppressed in the notation.
We write $d_\sigma(p,q)$ for geodesic distance, and use the closed caps
\[
B_{\rm ch}(q,t):=\{p:\|p-q\|\le t\},\qquad
B_{\rm geo}(q,r):=\{p:d_\sigma(p,q)\le r\}.
\]
The spherical-cap discrepancy of $\omega_d$ is defined by
\[
D_\infty(\omega_d):=
\sup_{\substack{q\in\Sphere^2\\0\le t\le2}}
\left|
\frac{\#(\omega_d\cap B_{\rm ch}(q,t))}{d}
-\frac{t^2}{4}
\right|.
\]
and the $L^2$ version by
\[
D_2(\omega_d):=
\left(
\int_{\Sphere^2}\int_0^2
\left|
\frac{\#(\omega_d\cap B_{\rm ch}(q,t))}{d}
-\frac{t^2}{4}
\right|^2
t\,\mathrm{d}t\,\mathrm{d}\sigma(q)
\right)^{1/2}.
\]
For $p\in\Sphere^2$, the distance from $p$ to the configuration
$\omega_d$ is
\[
\dist(p,\omega_d):=\min_{1\le j\le d}\|p-\hat{z}_j\|.
\]
The covering radius of $\omega_d$ is defined by
\[
\rho(\omega_d):=\max_{p\in\Sphere^2}\dist(p,\omega_d).
\]
For $d\ge2$, its separation distance is
\[
\delta(\omega_d):=\min_{1\le i<j\le d}\|\hat{z}_i-\hat{z}_j\|.
\]
The chordal Voronoi cell of $\hat{z}_i$ is
\[
V_i:=\{p\in\Sphere^2:\|p-\hat{z}_i\|\le\|p-\hat{z}_j\|
\text{ for every }j\}.
\]
The cell boundaries have $\sigma$-measure zero.

We write $A\lesssim B$ if $A\le CB$ for some constant $C>0$,
and $A\asymp B$ if $cB\le A\le CB$ for some constants $c,C>0$.
Unless otherwise specified, implicit constants are universal;
subscripts indicate allowed dependence on the corresponding parameters.

We write families as $\{\omega_d\}_{d\in\mathcal D}$ or
$\{f_d\}_{d\in\mathcal D}$, where $\mathcal D$ is an infinite
subset of $\N$. 
All asymptotic statements are understood along this index set.

\subsection{Logarithmic energy and log-balancedness}\label{sec:EFP}

Set
\[
\Energyminimal(d):=\min_{\hat{z}_1,\ldots,\hat{z}_d\in\Sphere^2}
\Energy(\hat{z}_1,\ldots,\hat{z}_d),
\qquad
M_d:=\exp\bigl(-\Energyminimal(d)\bigr).
\]
Thus $M_d$ is the maximum ordered product of pairwise distances. A configuration attaining this maximum, or equivalently minimizing $\Energy$, is called an \emph{elliptic Fekete configuration}. We retain the classical global problem:
\begin{problem}\label{prob:min-log-energy}
Give the asymptotic expansion of $\Energyminimal(d)$ as $d\to\infty$.
\end{problem}
The expansion available at the order needed here is
\begin{equation}\label{eq:asymp_min_energy}
\Energyminimal(d)=\kappa d^2-\frac12d\log d+C_{\log}d+o(d).
\end{equation}
The existence of $C_{\log}$ and its upper bound below follow from B\'etermin and Sandier \cite[Theorem~1.5]{BeterminSandier2018}. Together with Marzo's lower bound \cite{Marzo2025}, they give
\begin{equation}\label{eq:Clog-bounds}
-0.056845\ldots=C_{\rm lb}\le C_{\log}\le C_{\rm conj}=-0.055605\ldots.
\end{equation}
Here
\begin{equation*}
C_{\rm conj}
=2\log2+\frac12\log\frac23
+3\log\frac{\sqrt\pi}{\Gamma(1/3)},
\end{equation*}
where $\Gamma$ is the gamma function. 
The conjecture $C_{\log}=C_{\rm conj}$ is discussed in
\cite{BrauchartHardinSaff2012,BeterminSandier2018}.
A proof of this equality has recently been claimed in a
preprint released by OpenAI \cite{OpenAI2026Coulomb}.
The author has not independently verified the argument.
Our estimates use only the established upper bound
$C_{\log}\le C_{\rm conj}$ and do not rely on this equality.

For every configuration, we have
\begin{equation*}
\Energy(\omega_d)=-\sum_{i=1}^d\log P_i.
\end{equation*}
In the Fekete case, $\prod_iP_i=M_d$, and \cref{eq:asymp_min_energy} therefore fixes their geometric mean:
\begin{equation}\label{eq:fekete-geometric-mean}
M_d^{1/d}=\left(\prod_{i=1}^dP_i\right)^{1/d}
=\sqrt d\,e^{-\kappa d-C_{\log}+o(1)}.
\end{equation}
We use the term log-balancedness for the following uniform
control of the nodal products.

\begin{definition}\label{def:log-balanced}
A family $\{\omega_d\}_{d\in\mathcal D}$ of spherical
configurations is \emph{log-balanced} if there exist constants
$c,C>0$, independent of $d$ and $i$, such that
\[
c\sqrt d\,e^{-\kappa d}
\le P_i\le
C\sqrt d\,e^{-\kappa d},
\qquad d\in\mathcal D,\quad 1\le i\le d.
\]
\end{definition}

For a log-balanced family, taking logarithms in the definition
and summing over the nodes gives
\[
\Energy(\omega_d)
=-\sum_{i=1}^d\log P_i
=\kappa d^2-\frac12d\log d+O(d).
\]
Thus its logarithmic energy has the same first two terms
as the minimum in \cref{eq:asymp_min_energy}, and
\[
0\le \Energy(\omega_d)-\Energyminimal(d)\lesssim d.
\]
The implicit constants depend only on the log-balancedness
bounds of the family. This estimate does not determine
the coefficient of the linear term.

The configurations constructed in \cite{BeltranEtayoMarzoOrtegaCerda} satisfy this condition.
Our study of log-balancedness is motivated by two main problems.

\begin{problem}[Existence]\label{prob:log-balanced-existence}
Does there exist a log-balanced family of elliptic Fekete
configurations?
\end{problem}

\begin{problem}[Universality]\label{prob:uniform-products-conditioning}
Is every family of elliptic Fekete configurations log-balanced?
\end{problem}

We obtain uniform estimates and equivalent criteria for
Fekete families, but do not resolve either problem.

\subsection{Spherical points and polynomials}\label{sec:sphericalpoints}

Let $N:=(0,0,1)$ be the north pole. 
Stereographic projection from $N$ is given by
\[
\pi:\Sphere^2\setminus\{N\}\longrightarrow\C,
\qquad
\pi(x_1,x_2,x_3)=\frac{x_1+\mathrm i x_2}{1-x_3},
\]
with inverse
\begin{equation}\label{eq:inverse-stereographic-projection}
\hat z=\frac{1}{1+|z|^2}
\bigl(2\operatorname{Re}z,2\operatorname{Im}z,|z|^2-1\bigr).
\end{equation}
In particular,
\begin{equation*}
\|\hat z-\hat w\|=
\frac{2|z-w|}{\sqrt{1+|z|^2}\sqrt{1+|w|^2}}.
\end{equation*}
A configuration $\omega_d=\{\hat{z}_1,\ldots,\hat{z}_d\} \subset\Sphere^2$ avoiding $N$ determines a monic polynomial with simple roots,
\begin{equation*}
f_{\omega_d}(z):=\prod_{j=1}^d(z-z_j),
\qquad z_j:=\pi(\hat{z}_j),
\end{equation*}
and this is a bijection between such unordered configurations and monic degree-$d$ polynomials with simple roots. 

To include configurations containing $N$, we identify $\Sphere^2$ with $\mathbb{PC}^1$ and use homogeneous polynomials. 
If $\hat{z}_j$ corresponds to $[u_j:v_j]$, choose a representative $(u_j,v_j)\ne(0,0)$ and define
\begin{equation*}
F_{\omega_d}(U,V):=\prod_{j=1}^d(u_jV-v_jU).
\end{equation*}
This is a homogeneous polynomial of degree $d$ whose zeros in
$\mathbb{PC}^1$ are precisely the points corresponding to $\omega_d$.
Choosing different representatives $(u_j,v_j)$ only multiplies
$F_{\omega_d}$ by a nonzero constant.
When no point is at infinity, $F_{\omega_d}(U,V)=U^df_{\omega_d}(V/U)$ after choosing $[u_j:v_j]=[1:z_j]$. Every rotation $R$ of $\Sphere^2$ is induced on the projective line by a matrix $A\in\mathrm{SU}(2)$. A homogeneous polynomial
representing the rotated configuration $R\omega_d$ is
\[
F_R:=F_{\omega_d}\circ A^{-1}.
\]
Since $\omega_d$ is finite, we can choose $R$ so that
$N\notin R\omega_d$. Then $F_R(0,1)\ne0$, and the corresponding
monic polynomial is
\[
f_{R\omega_d}(z)=\frac{F_R(1,z)}{F_R(0,1)}.
\]

\subsection{Shub--Smale condition number}\label{sec:condition_number}

For a polynomial $f(z)=\sum_{k=0}^d a_kz^k$ of degree $d$, the Bombieri--Weyl norm is defined as
\[
\|f\|_{\mathrm{BW}}^2:=\sum_{k=0}^d\frac{|a_k|^2}{\binom dk}.
\]
It is the unitarily invariant norm of the homogeneous form $U^df(V/U)$. The normalized Shub--Smale condition number measures the first-order sensitivity of a simple projective root to relative coefficient perturbations in this norm. 
At a finite simple root $z_i$, it is
\begin{equation}\label{eq:condition-bombieri}
\cond{f}{z_i}
=\frac{\sqrt d\,(1+|z_i|^2)^{(d-2)/2}}{|f'(z_i)|}
\,\|f\|_{\mathrm{BW}};
\end{equation}
see \cite{ShubSmale1993}. We use $\mu$ only for this normalized condition number and set
\[
\condtotal{f}:=\max_{1\le i\le d}\cond{f}{z_i}.
\]
The bridge to spherical geometry is the exact identity
\begin{equation}\label{eq:norm_cond}
\cond{f}{z_i}
=\frac{\sqrt{d(d+1)}}2\frac{I_d^{1/2}}{P_i},
\end{equation}
proved in \cite[Proposition~2.1]{Etayo2021}; see also \cite[Proposition~2]{ShubSmale1993}. 
Rotations preserve chordal distances and the measure $\sigma$,
so they leave $P_i$ and $I_d$ unchanged. By \cref{eq:norm_cond}, the local condition numbers at corresponding roots are therefore unchanged under rotations, the numerator is common to every root, and
\begin{equation*}
\frac{\cond{f}{z_i}}{\cond{f}{z_j}}=\frac{P_j}{P_i}.
\end{equation*}
There is a universal $c_0>0$ such that every degree-$d$ polynomial with simple roots satisfies
\begin{equation*}
\condtotal{f}\ge c_0\sqrt d;
\end{equation*}
see \cite[Lemma~1.1]{BeltranEtayoMarzoOrtegaCerda} and \cite{Etayo2021} for an explicit asymptotic lower bound. Conversely, \cite{BeltranEtayoMarzoOrtegaCerda} constructs polynomials $f_d$ with
\begin{equation*}
\condtotal{f_d}\le C_0\sqrt d
\end{equation*}
for a universal $C_0$. Thus $\sqrt d$ is the optimal order of the \emph{global} condition number. 
Shub and Smale \cite{ShubSmale1993} proved that degree-$d$
polynomials whose roots correspond to elliptic Fekete
configurations on $\Sphere^2$ have global condition number $O(d)$.

\subsection{Wagner's logarithmic polarization}\label{sec:Wagners}

With the notation of \cref{eq:products-definition}, Wagner's logarithmic polarization problem consists in minimizing
$\mathcal Q(\omega_d)$ over all configurations of $d$ points on
$\Sphere^2$. We denote the minimum by
\begin{equation*}
\mathcal Q_d
:=\min_{\hat{z}_1,\ldots,\hat{z}_d\in\Sphere^2}
\mathcal Q(\omega_d),
\qquad
\omega_d=\{\hat{z}_1,\ldots,\hat{z}_d\} \subset\Sphere^2.
\end{equation*}
Repeated points are allowed in this minimization. 
Wagner's estimates \cite[Theorems~1 and~2]{Wagner1989} give absolute constants $c_1,c_2>0$ such that
\begin{equation*}
(1+c_1)e^{-\kappa d}\le\mathcal Q_d\le(1+c_2)e^{-\kappa d}.
\end{equation*}
For Wagner's problem, we consider asymptotic optimality at
the exponential scale, expressed by
\[
\log\frac{\mathcal Q(\omega_d)}{\mathcal Q_d}=o(d)
\qquad(d\to\infty).
\]
This condition does not require
$\mathcal Q(\omega_d)/\mathcal Q_d\to1$.

\subsection{Scope and organization}\label{sec:organization}

We obtain separation bounds, nodewise Riesz-potential estimates,
and comparisons between nodal products and integral norms for
elliptic Fekete configurations. 
An inverse-fourth-power bound for local logarithmic minimizers
improves the separation estimate of Dragnev \cite{Dragnev2002}.
Combined with covering, it yields a further improvement
for elliptic Fekete configurations
(\cref{th:local-minimum-inverse-fourth,cor:fekete-separation}).
The product and integral comparisons also yield quantitative
bounds for logarithmic polarization and local conditioning.

The local counting and covering estimates are taken from
Lev and Ortega-Cerd\`a
\cite[Theorem~1 and Lemma~5]{LevOrtegaCerda2016}.
For $0<s<2$, Leopardi \cite[Corollary~1.2]{Leopardi2013}
bounds the Riesz-energy error in terms of discrepancy and
separation. Applied to Fekete configurations with
$D_\infty(\omega_d)=O(d^{-1/2})$, his result gives a global
remainder $O_s(d^{3/2+s/4})$.
We improve this to $O_s(d^{1+s/2})$, obtain potential estimates
uniform at every node, and treat the critical exponent $s=2$
(\cref{thm:fekete-riesz-potentials}).
We further give explicit bounds for the normalized second-order
remainders, including strictly negative upper bounds
(\cref{cor:fekete-riesz-second-order}).
For every fixed $0<s<2$, these upper bounds also apply to the
minimal Riesz energy and strictly improve the second-order constant
of de la Torre and Marzo \cite[Corollary~4.1]{deLaTorreMarzo2025};
see \cref{prop:minimal-riesz-upper-comparison}.
The lower constants recover those of Brauchart
\cite[Theorems~3 and~4]{Brauchart2026}.

The global Riesz estimates extend to spherical roots of
polynomials with condition number $O(\sqrt d)$
(\cref{th:optimal-conditioning-riesz}).
Both classes are treated by \cref{lem:log-to-riesz-transfer},
which combines a logarithmic-energy bound with control of
close pairs. For Fekete configurations, second variation
provides this control; for well-conditioned polynomials,
we combine \cite[Theorem~1.5]{Etayo2021} with
\cref{lem:conditioning-close-pairs}, without assuming
logarithmic minimality or separation of order $d^{-1/2}$.
For Fekete families, we establish equivalent criteria for
log-balancedness and derive its consequences for integral
norms and local conditioning. The existence and universality
questions in
\cref{prob:log-balanced-existence,prob:uniform-products-conditioning}
remain unresolved here.
For general polynomials, our analysis of the construction in
\cite{BeltranLizarte2022} shows that optimal-order global
conditioning alone does not imply log-balancedness.

The statements are collected in \cref{sec:main_results}:
\cref{sec:prop_EFP} treats separation and Riesz estimates;
\cref{sec:products-integrals} treats products and integrals;
\cref{sec:polarization-results} treats logarithmic polarization;
and \cref{sec:logbalanced} treats conditioning and log-balancedness.
\Cref{sec:EFP_FC} recalls the interpretation of elliptic Fekete
configurations on the projective line.
All proofs are given in \cref{sec:proofs}.
\section{Main results}\label{sec:main_results}

\subsection{Geometry and Riesz estimates}\label{sec:prop_EFP}

\subsubsection{Geometric inputs (Lev--Ortega-Cerd\`a)}
We recall the spherical forms of the local counting and covering
estimates in \cite[Theorem~1 and Lemma~5]{LevOrtegaCerda2016},
together with their discrepancy consequences.

\begin{theorem}\label{thm:fekete-cap-discrepancy}
There is a universal $C>0$ such that every elliptic Fekete configuration $\omega_d\subset\Sphere^2$, $d\ge2$, satisfies
\begin{equation}\label{eq:fekete-local-count}
\left|\#(\omega_d\cap B_{\rm geo}(q,r))
-d\,\sigma(B_{\rm geo}(q,r))\right|
\le C(1+r\sqrt d),
\end{equation}
uniformly for $q\in\Sphere^2$ and $0\le r\le\pi$. 
In particular,
\begin{equation}\label{eq:fekete-cap-discrepancy}
D_\infty(\omega_d)\le Cd^{-1/2}.
\end{equation}
\end{theorem}

\nopagebreak[4]
\noindent
See the \hyperref[proof:fekete-cap-discrepancy]{proof}
on page~\pageref{proof:fekete-cap-discrepancy}.

\begin{theorem}\label{th:covering-radius-fekete}
There is a universal $C>0$ such that every elliptic Fekete configuration satisfies
\[
\frac2{\sqrt d}\le\rho(\omega_d)\le\frac C{\sqrt d},
\qquad d\ge2.
\]
\end{theorem}

\nopagebreak[4]
\noindent
See the \hyperref[proof:covering-radius-fekete]{proof}
on page~\pageref{proof:covering-radius-fekete}.

\begin{corollary}\label{cor:fekete-discrepancy-distance-energy}
Every elliptic Fekete configuration satisfies
\[
D_2(\omega_d)\lesssim d^{-1/2},
\qquad
0\le\frac43d^2-E_{-1}(\omega_d)\lesssim d,
\qquad d\ge2,
\]
with universal constants.
\end{corollary}

\nopagebreak[4]
\noindent
See the \hyperref[proof:fekete-discrepancy-distance-energy]{proof}
on page~\pageref{proof:fekete-discrepancy-distance-energy}.

The first assertion follows from $D_2\le\sqrt2 D_\infty$, and the second combines this bound with Stolarsky's invariance principle \cite{Stolarsky1973}.

\subsubsection{Separation and Riesz estimates}
Separation estimates for elliptic Fekete configurations go back to
\cite{RakhmanovSaffZhou1994}; Dragnev \cite[Theorem~2]{Dragnev2002}
proved $\delta(\omega_d)\ge2/\sqrt{d-1}$. The next result uses
second variation and requires only local minimality.

\begin{samepage}
\begin{theorem}\label{th:local-minimum-inverse-fourth}
Let $\omega_d\subset\Sphere^2$ be a configuration of $d\ge2$
distinct points that locally minimizes $\Energy$ on $(\Sphere^2)^d$.
Then
\begin{gather}
U_{4,i}(\omega_d)\le\frac{(d-1)^2}{16},
\qquad 1\le i\le d,
\label{eq:local-minimum-inverse-fourth}\\
\delta(\omega_d)\ge\frac{2}{(d^2-3d+3)^{1/4}}.
\label{eq:local-minimum-separation}
\end{gather}
\end{theorem}

\nopagebreak[4]
\noindent
See the \hyperref[proof:local-minimum-inverse-fourth]{proof}
on page~\pageref{proof:local-minimum-inverse-fourth}.
\end{samepage}

The separation bound is strictly stronger than Dragnev's for
$d>2$, but its leading constant is still $2$. Combining
\cref{eq:local-minimum-inverse-fourth} with the covering estimate
gives a uniform improvement of this constant for Fekete configurations.

\begin{corollary}\label{cor:fekete-separation}
There is a universal constant $\alpha>2$ such that every elliptic
Fekete configuration satisfies
\begin{equation*}
\delta(\omega_d)\ge\frac{\alpha}{\sqrt d},\qquad d\ge2.
\end{equation*}
\end{corollary}

\nopagebreak[4]
\noindent
See the \hyperref[proof:fekete-separation]{proof}
on page~\pageref{proof:fekete-separation}.

This answers the question raised in \cite[Remark~3]{Dragnev2002}
of improving the asymptotic separation constant $2$, without
identifying the optimal constant. To our knowledge, the
inverse-fourth-power estimate and these separation improvements
have not appeared in the literature.

We next integrate the localized counting error against a Riesz kernel. This is a standard method; see G\"otz \cite[Proposition~7]{Gotz2003} for the radial counting representation of discrete Riesz energy. Keeping the radius dependence in \cref{eq:fekete-local-count} gives the nodewise estimates below. Their global counterpart has a sharper remainder for $0<s\le1$: it uses the logarithmic-energy and inverse-fourth-power bounds through the positive definite kernels of \cref{lem:log-to-riesz-transfer}, rather than summing the nodewise error bounds.

\begin{theorem}\label{thm:fekete-riesz-potentials}
Let $\omega_d=\{\hat{z}_1,\ldots,\hat{z}_d\}\subset\Sphere^2$ be an elliptic Fekete configuration, $d\ge2$. For each fixed $0<s<2$, define
\[
\mathcal R_s(d):=
\begin{cases}
\sqrt d,&0<s<1,\\
\sqrt d\,\log d,&s=1,\\
d^{s/2},&1<s<2.
\end{cases}
\]
Then, uniformly for $1\le i\le d$,
\begin{equation}\label{eq:fekete-local-riesz-subcritical}
U_{s,i}(\omega_d)=\frac{2^{1-s}}{2-s}\,d+O_s\bigl(\mathcal R_s(d)\bigr),
\qquad 0<s<2,
\end{equation}
and at the critical exponent,
\begin{equation}\label{eq:fekete-local-riesz-two}
U_{2,i}(\omega_d)=\frac14\,d\log d+O(d).
\end{equation}
The global energies satisfy
\begin{align}
E_s(\omega_d)&=\frac{2^{1-s}}{2-s}\,d^2
+O_s\bigl(d^{1+s/2}\bigr),&&0<s<2,
\label{eq:fekete-riesz-subcritical}\\
E_2(\omega_d)&=\frac14\,d^2\log d+O(d^2).
\label{eq:fekete-riesz-two}
\end{align}
The implicit constants are independent of $d$, the node $i$,
and the Fekete configuration. For $0<s<2$, they may depend on $s$.
\end{theorem}

\nopagebreak[4]
\noindent
See the \hyperref[proof:fekete-riesz-potentials]{proof}
on page~\pageref{proof:fekete-riesz-potentials}.

\begin{remark}\label{rem:riesz-scope}
For $0<s<2$, \cref{eq:fekete-local-riesz-subcritical} gives convergence of $d^{-1}U_{s,i}$ to the continuous potential in \cref{eq:continuous-riesz-potential}, uniformly in the node. The logarithm in the error at $s=1$ comes from estimating the counting error by its absolute value; no optimality of that remainder is claimed.  
\end{remark}

To state the second-order bounds, define the normalized remainder
\begin{equation*}
R_s(\omega_d):=
\begin{cases}
\displaystyle
\frac{E_s(\omega_d)-\frac{2^{1-s}}{2-s}d^2}{d^{1+s/2}},
&0<s<2,\\[6pt]
\displaystyle
\frac{E_2(\omega_d)-\frac14d^2\log d}{d^2},&s=2.
\end{cases}
\end{equation*}
Let $\gamma$ denote Euler's constant and, for $A\in\R$, set
\begin{equation*}
\eta(A):=2A+\gamma+1-2\log2.
\end{equation*}
For $0<s<2$ and $B>0$, write $a=s/2$ and define
\begin{equation}\label{eq:transfer-second-order-constants}
\begin{aligned}
c_s^-&:=-\frac{2^{-s}}{\Gamma(a)a(1-a)},\\
c_s^+(A,B)&:=c_s^-+
\frac{2^{-s}}{\Gamma(a)}\bigl(\eta(A)+4B\bigr).
\end{aligned}
\end{equation}
For $t\ge1$, let $J(t):=\int_t^\infty e^{-u}/u\,\mathrm{d}u$ and set
\begin{equation}\label{eq:critical-transfer-function}
\Phi_A(t):=\frac14\Bigl[\gamma+\log t+1
+t\bigl(\eta(A)+\log t-2\bigr)+e^{-t}-tJ(t)\Bigr].
\end{equation}

\begin{corollary}\label{cor:fekete-riesz-second-order}
Let $\{\omega_d\}_{d\in\mathcal D}$ be a family of elliptic
Fekete configurations. For each fixed $0<s<2$,
\begin{align}
\liminf_{d\to\infty}
R_s(\omega_d)
&\ge c_s^-,\label{eq:fekete-second-order-lower}\\
\limsup_{d\to\infty}
R_s(\omega_d)
&\le c_s^+\!\left(C_{\rm conj},\frac1{4e}\right)<0.
\label{eq:fekete-second-order-upper}
\end{align}
At $s=2$,
\begin{align}
\liminf_{d\to\infty}
R_2(\omega_d)
&\ge\frac{\gamma-1}{4},\label{eq:fekete-critical-second-order-lower}\\
\limsup_{d\to\infty}
R_2(\omega_d)
&\le\inf_{t\ge1}\Phi_{C_{\rm conj}}(t)
\le c_{\rm up}<0,\label{eq:fekete-critical-second-order-upper}
\end{align}
where $c_{\rm up}:=\Phi_{C_{\rm conj}}(4/3)
=-0.054966\ldots$.
All bounds are independent of the chosen Fekete family.
\end{corollary}

\nopagebreak[4]
\noindent
See the \hyperref[proof:fekete-riesz-second-order]{proof}
on page~\pageref{proof:fekete-riesz-second-order}.

The lower bounds follow from positivity of the kernels and hold
for arbitrary configurations. They recover the second-order constants
in Brauchart's bounds \cite[Theorems~3 and~4]{Brauchart2026}:
his subcritical constant satisfies
\[
-\frac{2^{1-s}}{(2-s)\Gamma(1+s/2)}
=-\frac{2^{-s}}{\Gamma(s/2)(s/2)(1-s/2)}=c_s^-.
\]
Brauchart's estimates also provide remainders
$O_s(d^{s/2-1})$ for $0<s<2$ and $O(1)$ for $s=2$.
Our upper bounds use the logarithmic-energy and
inverse-fourth-power estimates for Fekete configurations.

For $s>0$, denote the minimal Riesz energy by
\[
\mathcal E_s(d):=
\min_{\substack{\omega_d\subset\Sphere^2\\\#\omega_d=d}}
E_s(\omega_d).
\]
The preceding upper bounds also apply to this minimum.

\begin{proposition}\label{prop:minimal-riesz-upper-comparison}
For $0<s<2$, define
\[
C_{\rm F}(s):=c_s^+\!\left(C_{\rm conj},\frac1{4e}\right),
\qquad
C_{\rm TM}(s):=
2^{-s}\frac{s}{2}\left(1+\frac{s}{2}\right)
\Gamma\left(1-\frac{s}{2}\right)
\zeta\left(1-\frac{s}{2}\right),
\]
where $\zeta$ is the Riemann zeta function. Then
\begin{equation}\label{eq:minimal-riesz-second-order-upper}
\limsup_{d\to\infty}
\frac{\mathcal E_s(d)-\frac{2^{1-s}}{2-s}d^2}{d^{1+s/2}}
\le C_{\rm F}(s)<C_{\rm TM}(s).
\end{equation}
\end{proposition}

\nopagebreak[4]
\noindent
See the \hyperref[proof:minimal-riesz-upper-comparison]{proof}
on page~\pageref{proof:minimal-riesz-upper-comparison}.

The constant $C_{\rm TM}(s)$ is the upper second-order constant
in \cite[equation~(7) and Corollary~4.1]{deLaTorreMarzo2025}.
Thus \cref{eq:minimal-riesz-second-order-upper} strictly improves
that bound for every fixed $0<s<2$.
To our knowledge, $C_{\rm F}(s)$ is the best currently known
explicit upper bound for this normalized second-order remainder.
The underlying bound in \cref{cor:fekete-riesz-second-order}
holds uniformly over all elliptic Fekete configurations;
no convergence of the normalized remainder is asserted.

For example, at $s=1$ the lower and upper constants are
$-2/\sqrt\pi=-1.128379\ldots$ and
$c_1^+(C_{\rm conj},1/(4e))=-1.002116\ldots$.
At $s=2$, the upper estimate in particular gives
\[
E_2(\omega_d)\le\frac14d^2\log d+c_{\rm up}d^2+o(d^2),
\]
where the $o(d^2)$ term can be chosen independently of the Fekete configuration.
The constant conjectured for the minimal $2$-energy in
\cite[Conjecture~5]{BrauchartHardinSaff2012} satisfies the identity
proved in \cite[Theorem~1]{Brauchart2026}:
\begin{equation*}
C_{2,\rm conj}
=\frac{C_{\rm conj}+\gamma-\log2}{2}
=-0.085768\ldots.
\end{equation*}
Steinerberger \cite{Steinerberger2026} constructs configurations with $2$-energy at most $\frac14d^2\log d+C_{2,\rm conj}d^2+o(d^2)$.
That result does not assert logarithmic minimality, whereas \cref{cor:fekete-riesz-second-order} applies to every elliptic
Fekete family.

\begin{corollary}\label{cor:fekete-riesz-hypersingular}
Let $\omega_d\subset\Sphere^2$ be an elliptic Fekete
configuration with $d\ge2$. For each fixed $s>2$,
\[
U_{s,i}(\omega_d)\asymp_s d^{s/2},
\qquad 1\le i\le d,
\]
and
\[
E_s(\omega_d)\asymp_s d^{1+s/2}.
\]
The implicit constants depend only on $s$.
\end{corollary}

\nopagebreak[4]
\noindent
See the \hyperref[proof:fekete-riesz-hypersingular]{proof}
on page~\pageref{proof:fekete-riesz-hypersingular}.


\subsection{Products of distances and integral estimates}\label{sec:products-integrals}

The next comparison is the link between the geometric estimates and conditioning. 

\begin{theorem}
\label{th:fekete-discrete-Id-Ii}
Let $\omega_d=\{\hat{z}_1,\ldots,\hat{z}_d\}\subset\Sphere^2$ be an elliptic
Fekete configuration, with $d\ge2$.
Then
\[
    I_d\asymp \frac{\sum_{j=1}^d P_j^2}{d^2},
\]
and, uniformly for $1\le i\le d$,
\[
    I_{d,i}\asymp
    \frac{P_i^2}{d}
    +\frac{1}{d^2}
      \sum_{\substack{1\le j\le d\\j\ne i}}
      \frac{P_j^2}{\|\hat{z}_i-\hat{z}_j\|^2}.
\]
All implicit constants are universal.
\end{theorem}
\nopagebreak[4]
\noindent
See the \hyperref[proof:fekete-discrete-Id-Ii]{proof}
on page~\pageref{proof:fekete-discrete-Id-Ii}.

\begin{corollary}\label{cor:average-Ii-Id-log-gap}
For elliptic Fekete configurations with $d\ge2$,
\[
\sum_{i=1}^d I_{d,i}\asymp d\log d\,I_d,
\]
with universal constants. 
In particular, $I_{d,i}\lesssim I_d$ cannot hold uniformly in both $i$ and $d$.
\end{corollary}

\nopagebreak[4]
\noindent
See the \hyperref[proof:average-Ii-Id-log-gap]{proof}
on page~\pageref{proof:average-Ii-Id-log-gap}.

\begin{corollary}\label{cor:fekete-norm-comparisons}
For every elliptic Fekete configuration, $d\ge2$,
\[
I_{d,i}^{1/2}\le P_i\le\sqrt d\,I_{d,i}^{1/2},
\qquad 1\le i\le d,
\]
and
\[
\max_{1\le i\le d}P_i\gtrsim\sqrt d\,I_d^{1/2}.
\]
If $f$ has the corresponding roots, then
\[
\cond{f}{z_i}\ge\frac{\sqrt{d+1}}{2}
\frac{I_d^{1/2}}{I_{d,i}^{1/2}},\qquad 1\le i\le d.
\]
The implicit constant is universal.
\end{corollary}

\nopagebreak[4]
\noindent
See the \hyperref[proof:fekete-norm-comparisons]{proof}
on page~\pageref{proof:fekete-norm-comparisons}.

The geometric mean in \cref{eq:fekete-geometric-mean} has the fine scale $\sqrt d\,e^{-\kappa d}$. The next theorem controls each product and each integral, at the cost of explicit polynomial factors.

\begin{theorem}\label{th:fekete-exponential-scale-integrals}
Let $\omega_d=\{\hat{z}_1,\ldots,\hat{z}_d\}\subset\Sphere^2$
be an elliptic Fekete configuration with $d\ge2$.
Then, uniformly for $1\le i\le d$,
\[
e^{-\kappa(d-1)}
\le P_i
\lesssim d\,e^{-\kappa d},
\]
\[
e^{-\kappa(d-1)}
\le I_{d,i}^{1/2}
\lesssim \sqrt{d\log d}\,e^{-\kappa d},
\]
and
\[
e^{-\kappa d}
\le I_d^{1/2}
\lesssim \sqrt d\,e^{-\kappa d}.
\]
\end{theorem}

\nopagebreak[4]
\noindent
See the \hyperref[proof:fekete-exponential-scale-integrals]{proof}
on page~\pageref{proof:fekete-exponential-scale-integrals}.


\subsection{Logarithmic polarization}\label{sec:polarization-results}

The next theorem compares $\mathcal Q(\omega_d)$ with the
optimal value $\mathcal Q_d$ for elliptic Fekete configurations.

\begin{theorem}\label{th:fekete-wagner-polarization}
For every elliptic Fekete configuration $\omega_d\subset\Sphere^2$
with $d\ge2$,
\[
0\le
\log\frac{\mathcal Q(\omega_d)}{\mathcal Q_d}
\le\frac12\log d+O(1).
\]
\end{theorem}

\nopagebreak[4]
\noindent
See the \hyperref[proof:fekete-wagner-polarization]{proof}
on page~\pageref{proof:fekete-wagner-polarization}.

Thus every elliptic Fekete family is asymptotically optimal
for Wagner's problem at the exponential scale, since
\[
\log\frac{\mathcal Q(\omega_d)}{\mathcal Q_d}
=O(\log d)=o(d).
\]

\begin{proposition}\label{prop:fekete-polarization-max-product}
For every elliptic Fekete configuration with $d\ge2$,
\[
\mathcal Q(\omega_d)\asymp\frac1{\sqrt d}\max_{1\le i\le d}P_i,
\]
with universal comparison constants.
\end{proposition}

\nopagebreak[4]
\noindent
See the \hyperref[proof:fekete-polarization-max-product]{proof}
on page~\pageref{proof:fekete-polarization-max-product}.

\begin{proposition}\label{prop:fekete-polarization-upper-balance}
Let $\{\omega_d\}_{d\in\mathcal D}$ be a family of elliptic Fekete configurations with $d\ge2$. Then
\[
\mathcal Q(\omega_d)\asymp e^{-\kappa d}
\quad\Longleftrightarrow\quad
P_i\lesssim\sqrt d\,e^{-\kappa d}
\quad\text{uniformly for }1\le i\le d.
\]
The comparison constants are independent of $d$ and $i$, but may depend on the family.
\end{proposition}

\nopagebreak[4]
\noindent
See the \hyperref[proof:fekete-polarization-upper-balance]{proof}
on page~\pageref{proof:fekete-polarization-upper-balance}.


\subsection{Conditioning and log-balancedness}\label{sec:logbalanced}

We now study the relation between log-balancedness and
polynomial conditioning, distinguishing results for general
polynomials from those requiring the Fekete property.

\begin{proposition}\label{prop:optcond_Id}
Let $\{f_d\}_{d\in\mathcal D}$ be a family of degree-$d$ polynomials with simple roots. If $\condtotal{f_d}\le C\sqrt d$ with $C$ independent of $d$, then
\[
e^{-\kappa d}\le I_d^{1/2}\lesssim_C e^{-\kappa d}.
\]
\end{proposition}

\nopagebreak[4]
\noindent
See the \hyperref[proof:optcond_Id]{proof}
on page~\pageref{proof:optcond_Id}.

\begin{remark}\label{rem:optcond-bombieri}
At the level of asymptotic order, this is the integral form of \cite[Theorems~4.3 and~4.5]{Etayo2021}. For a monic polynomial, the identity in \cite[Theorem~4.2]{Etayo2021} is
\[
\frac{\prod_{j=1}^d\|z-z_j\|_{\mathrm{BW}}}{\|f_d\|_{\mathrm{BW}}}
=\frac{2^d}{\sqrt{d+1}\,I_d^{1/2}}.
\]
Thus the conclusion is equivalent to this norm ratio being comparable to $e^{d/2}/\sqrt{d+1}$, with constants depending at most on the conditioning bound.
\end{remark}

Optimal global conditioning also controls Riesz energies, not only the logarithmic case. 
The logarithmic-energy estimate of \cite[Theorem~1.5 and equation~(9)]{Etayo2021} is one input. To estimate Riesz energies, we also need to control the
contribution of nearby roots. We obtain this control from
the conditioning bound in \cref{lem:conditioning-close-pairs}.
Unlike the Fekete argument, this does not assume a separation bound of order $d^{-1/2}$.

\begin{theorem}
\label{th:optimal-conditioning-riesz}
Let $C_0>0$ and $d\ge2$. If $f_d$ is a degree-$d$ polynomial with simple roots and $\condtotal{f_d}\le C_0\sqrt d$, its spherical root configuration $\omega_d=\{\hat{z}_1,\ldots,\hat{z}_d\}\subset\Sphere^2$ satisfies
\begin{align*}
E_s(\omega_d)&=\frac{2^{1-s}}{2-s}\,d^2
+O_{C_0,s}\bigl(d^{1+s/2}\bigr),&&0<s<2,
\\
E_2(\omega_d)&=\frac14\,d^2\log d+O_{C_0}(d^2).
\end{align*}
The constants are independent of $d$ and the polynomial. 
\end{theorem}
\nopagebreak[4]
\noindent
See the \hyperref[proof:optimal-conditioning-riesz]{proof}
on page~\pageref{proof:optimal-conditioning-riesz}.

\begin{theorem}\label{th:local-condition-implies-log-balanced}
Let $\{f_d\}_{d\in\mathcal D}$ be a family of complex polynomials with simple roots, written as
\[
f_d(z)=a_d\prod_{i=1}^d(z-z_i).
\]
Let $\{\omega_d\}_{d\in\mathcal D}$ be their spherical root
configurations, with
$\omega_d=\{\hat{z}_1,\ldots,\hat{z}_d\}\subset\Sphere^2$.
If there exist constants $0<c\le C<\infty$ such that, for every
$d\in\mathcal D$,
\[
c\sqrt d
\le
\cond{f_d}{z_i}
\le
C\sqrt d,
\qquad 1\le i\le d,
\]
then $\{\omega_d\}_{d\in\mathcal D}$ is log-balanced.
\end{theorem}
\nopagebreak[4]
\noindent
See the \hyperref[proof:local-condition-implies-log-balanced]{proof}
on page~\pageref{proof:local-condition-implies-log-balanced}.
\begin{corollary}\label{cor:log-balanced-local-lower-bound}
Let $\{\omega_d\}_{d\in\mathcal D}$ be log-balanced, and let
$\{f_d\}_{d\in\mathcal D}$ be any family of polynomials with
the corresponding spherical roots. Then, uniformly in $d$
and $1\le i\le d$,
\[
\cond{f_d}{z_i}
\asymp
\sqrt d\,e^{\kappa d}I_d^{1/2}.
\]
In particular,
\[
\sqrt d\lesssim\cond{f_d}{z_i},
\qquad d\in\mathcal D,\quad 1\le i\le d.
\]
All implicit constants depend only on the log-balancedness
bounds of the family.
\end{corollary}
\nopagebreak[4]
\noindent
See the \hyperref[proof:log-balanced-local-lower-bound]{proof}
on page~\pageref{proof:log-balanced-local-lower-bound}.

For Fekete configurations the discrete integral formula provides a particularly simple interpretation through weights.

\begin{proposition}\label{prop:condition-weights}
Let $f$ be a degree-$d$ polynomial with simple roots, $d\ge2$, and put
\[
\mu_i:=\cond{f}{z_i},
\qquad
w_i:=\frac{P_i^2}{\sum_{j=1}^dP_j^2}.
\]
For every such polynomial,
\begin{equation}\label{eq:condition-weight-identity}
w_i=\frac{\mu_i^{-2}}{\sum_{j=1}^d\mu_j^{-2}}.
\end{equation}
If its spherical root configuration is elliptic Fekete, then
\begin{equation}\label{eq:fekete-condition-weights}
\sum_{j=1}^d\mu_j^{-2}\asymp1,
\qquad
\mu_i^2\asymp\frac1{w_i},
\end{equation}
with universal constants. Consequently, for a Fekete family $\{\omega_d\}_{d\in\mathcal D}$ with $d\ge2$,
\begin{equation}\label{eq:global-condition-weights}
\condtotal{f_d}=O(\sqrt d)
\quad\Longleftrightarrow\quad
\min_{1\le i\le d}w_i\gtrsim\frac1d.
\end{equation}
The constants in this family equivalence may depend on the family.
\end{proposition}

\nopagebreak[4]
\noindent
See the \hyperref[proof:condition-weights]{proof}
on page~\pageref{proof:condition-weights}.

The following proposition collects the equivalent criteria for
log-balancedness within the Fekete class.

\begin{proposition}\label{prop:fekete-log-balanced-deletion}
Let $\{\omega_d\}_{d\in\mathcal D}$ be a family of elliptic Fekete
configurations, and let $\{f_d\}_{d\in\mathcal D}$ be polynomials
with the corresponding spherical roots. With the weights of
\cref{prop:condition-weights},
\begin{equation}\label{eq:balanced-condition-weights}
\begin{aligned}
&\{\omega_d\}_{d\in\mathcal D}\text{ is log-balanced}\\
&\quad\Longleftrightarrow\quad w_i\asymp d^{-1}\ \text{uniformly in }i,d\\
&\quad\Longleftrightarrow\quad \cond{f_d}{z_i}\asymp\sqrt d\ \text{uniformly in }i,d.
\end{aligned}
\end{equation}
These conditions are also equivalent to the existence of $C>0$,
independent of $d,i,j$, such that
\[
\left|
\Energy\!\left(\omega_d\setminus\{\hat z_i\}\right)
-\Energy\!\left(\omega_d\setminus\{\hat z_j\}\right)
\right|\le C,
\qquad 1\le i,j\le d.
\]
The constants in these family equivalences may depend on the family.
\end{proposition}

\nopagebreak[4]
\noindent
See the \hyperref[proof:fekete-log-balanced-deletion]{proof}
on page~\pageref{proof:fekete-log-balanced-deletion}.

Optimal global conditioning therefore gives a lower bound for every weight, whereas log-balancedness requires two-sided comparability with $1/d$. The condition $\sum_iw_i=1$ alone does not bridge this gap. 

There is nevertheless an unconditional restriction on the concentration of the Fekete weights. 
The next proposition bounds the ratio between the largest
and smallest nodal products and gives a lower bound for
the condition number at every root.

\begin{proposition}\label{prop:fekete-oscillation-local-condition}
There are universal constants $c,C>0$ such that every elliptic Fekete configuration $\omega_d=\{\hat{z}_1,\ldots,\hat{z}_d\}\subset\Sphere^2$, $d\ge2$, satisfies
\begin{equation*}
\frac{\max_{1\le i\le d}P_i}{\min_{1\le i\le d}P_i}\le C\sqrt d.
\end{equation*}
If $f_d$ is a polynomial with the corresponding spherical roots, then
\begin{equation}\label{eq:fekete-local-condition-log}
\cond{f_d}{z_i}\ge c\sqrt{\log(d+1)},
\qquad 1\le i\le d.
\end{equation}
\end{proposition}
\nopagebreak[4]
\noindent
See the \hyperref[proof:fekete-oscillation-local-condition]{proof}
on page~\pageref{proof:fekete-oscillation-local-condition}.

In the notation of \cref{prop:condition-weights}, this also yields
\[
w_i\lesssim\frac1{\log(d+1)}.
\]
The constants are universal. 
The existence of a log-balanced Fekete family is not established here; see \cref{prob:log-balanced-existence}.
The following corollary describes the integral estimates and
local conditioning that such a family would satisfy.

\begin{corollary}\label{cor:log-balanced+Fekete-local-lower-bound}
Suppose that $\{\omega_d\}_{d\in\mathcal D}$ is a log-balanced
family of elliptic Fekete configurations with $d\ge2$.
Let $\{f_d\}_{d\in\mathcal D}$ be any family of polynomials
with the corresponding spherical roots. Then
\[
I_d^{1/2}\asymp e^{-\kappa d},
\qquad
I_{d,i}^{1/2}\asymp\sqrt{\log d}\,e^{-\kappa d},
\qquad
\cond{f_d}{z_i}\asymp\sqrt d,
\]
uniformly in $d$ and $1\le i\le d$. The constants may depend on the log-balancedness bounds of the family, but not on the node or degree.
\end{corollary}

\nopagebreak[4]
\noindent
See the \hyperref[proof:log-balanced+Fekete-local-lower-bound]{proof}
on page~\pageref{proof:log-balanced+Fekete-local-lower-bound}.

Thus the logarithmic gap in \cref{cor:average-Ii-Id-log-gap} holds at every node under log-balancedness, not only after averaging.

Optimal-order global conditioning does not imply log-balancedness; the lower local bound in \cref{th:local-condition-implies-log-balanced} cannot be omitted.
The following counterexample is obtained by analyzing the first
parallel in the construction of Beltr\'an and Lizarte \cite{BeltranLizarte2022}.

\begin{theorem}\label{thm:optimal-condition-not-log-balanced}
There exists a family
$\{f_d\}_{d\in\mathcal D}$ of degree-$d$ polynomials with
simple roots, where
\[
\mathcal D:=\{4M^2:M\in\N,\ M\ge5\},
\]
such that
\[
\condtotal{f_d}
\le
\frac{19}{2}\sqrt{d+1},
\]
but the associated family of spherical root configurations
$\{\omega_d\}_{d\in\mathcal D}$ is not log-balanced. Indeed, there exist indices $i_M$ and a universal constant $c>0$ such that
\[
\frac{P_{i_M}}{\sqrt{d}\,e^{-\kappa d}}
\ge
cd^{1/6}\longrightarrow\infty.
\]
Moreover, the corresponding local condition numbers satisfy
\[
\cond{f_d}{z_{i_M}}\lesssim d^{1/3}.
\]
\end{theorem}
\nopagebreak[4]
\noindent
See the \hyperref[proof:optimal-condition-not-log-balanced]{proof}
on page~\pageref{proof:optimal-condition-not-log-balanced}.


\section{\texorpdfstring{Fekete configurations on $\mathbb{PC}^1$}{Fekete configurations on the projective line}}\label{sec:EFP_FC}

We recall the line-bundle formulation of Fekete points on the Riemann
sphere. We identify \(\mathbb{PC}^1\) with \(\mathbb C\cup\{\infty\}\) by writing
\([u:v]\in\mathbb{PC}^1\) and using the affine coordinate \(z=v/u\) on the chart
\(u\neq 0\); the point \([0:1]\) corresponds to \(z=\infty\). Composing this identification with the inverse stereographic
projection in \cref{eq:inverse-stereographic-projection}
gives the identification \(\mathbb{PC}^1\simeq\Sphere^2\),
explicitly
\[
[u:v]\longmapsto
\frac{1}{|u|^2+|v|^2}
\bigl(
2\operatorname{Re}(\overline u v),
2\operatorname{Im}(\overline u v),
|v|^2-|u|^2
\bigr).
\]
Let \(\mathcal O(1)\) denote the hyperplane line bundle over
\(\mathbb{PC}^1\), and set
\(\mathcal O(k):=\mathcal O(1)^{\otimes k}\).
For \(k\ge0\), the space of global holomorphic sections
\[
H^0(\mathbb{PC}^1,\mathcal O(k))
\]
is naturally identified with the space of homogeneous
polynomials of degree \(k\) in two variables \(U,V\), with basis
\[
S_\ell(U,V)=U^{k-\ell}V^\ell,\qquad \ell=0,\ldots,k.
\]
We endow \(\mathcal O(1)\) with the Fubini--Study Hermitian metric, whose induced metric on \(\mathcal O(k)\) is denoted by \(k\phi_{\rm FS}\); in homogeneous coordinates this means that, for a homogeneous polynomial \(S\) of degree \(k\),
\[
|S([u:v])|_{k\phi_{\rm FS}}
=
\frac{|S(u,v)|}{\bigl(|u|^2+|v|^2\bigr)^{k/2}}.
\]
Here $|S(u,v)|$ denotes the usual complex modulus; homogeneity
ensures that the quotient is independent of the chosen
representative of $[u:v]$.
For $A\in\mathrm{SU}(2)$, set $S_A:=S\circ A^{-1}$ and
write $Ap=[A(u,v)]$ when $p=[u:v]$.
Since $A$ preserves the Euclidean norm on $\C^2$,
\begin{equation}\label{eq:fubini-study-unitary-invariance}
|S_A(Ap)|_{k\phi_{\rm FS}}
=
|S(p)|_{k\phi_{\rm FS}}.
\end{equation}
The induced rotation of $\Sphere^2$ preserves $\sigma$.
Consequently, this simultaneous transformation of the section
and its argument preserves its supremum and $L^2(\sigma)$ norms.

A Fekete configuration for \((\mathcal O(k),k\phi_{\rm FS})\) is a set of \(d=k+1=\dim H^0(\mathbb{PC}^1,\mathcal O(k))\) points maximizing the Fubini--Study weighted Vandermonde determinant 
\[
\left|
\det\bigl(S_\ell(p_j)\bigr)_{\ell=0,\ldots,k;\,j=1,\ldots,d}
\right|_{k\phi_{\rm FS}}.
\]
This notion is independent of the chosen basis, since a change of basis only
multiplies the determinant by a non-zero constant.

\begin{lemma}\label{lem:elliptic-fekete-as-line-bundle-fekete}
Let $\omega_d=\{\hat{z}_1,\ldots,\hat{z}_d\}\subset\Sphere^2$,
and let $p_j\in\mathbb{PC}^1$ correspond to $\hat{z}_j$
for $j=1,\ldots,d$. Set $k=d-1$.
Then $\omega_d$ is an elliptic Fekete configuration if and
only if $\{p_1,\ldots,p_d\}$ is a Fekete configuration for
$(\mathcal O(k),k\phi_{\rm FS})$.
\end{lemma}

\nopagebreak[4]
\noindent
See the \hyperref[proof:elliptic-fekete-as-line-bundle-fekete]{proof}
on page~\pageref{proof:elliptic-fekete-as-line-bundle-fekete}.

In the affine coordinate $z$, the Fubini--Study weight is
$\phi_{\rm FS}(z)=\frac12\log(1+|z|^2)$, and
\[
i\partial\bar\partial\phi_{\rm FS}
=\frac{i\,dz\wedge d\bar z}{2(1+|z|^2)^2}>0.
\]
Thus $(\mathcal O(1),\phi_{\rm FS})$ is positive.
The bundle
\(\mathcal O(k)=\mathcal O(1)^{\otimes k}\) carries the induced metric
\(k\phi_{\rm FS}\), whose curvature is \(k\,i\partial\bar\partial\phi_{\rm FS}\).

\section{Proofs}\label{sec:proofs}

\subsection{Technical lemmas}\label{sec:proofs:elementary}

We first verify the line-bundle correspondence, then isolate three ingredients used throughout: one-point extremality, the mean of the logarithmic kernel, and the norm comparison for a section on the Riemann sphere.

\begin{proof}[Proof of \cref{lem:elliptic-fekete-as-line-bundle-fekete}]
\leavevmode\phantomsection
\label{proof:elliptic-fekete-as-line-bundle-fekete}
Since the ordered product in the definition of \(M_d\) is the square of the
unordered product, maximizing
\[
\prod_{\substack{1\le i,j\le d\\ i\neq j}}
\|\hat{z}_i-\hat{z}_j\|
\]
is equivalent to maximizing
\[
\prod_{1\le i<j\le d}\|\hat{z}_i-\hat{z}_j\|.
\]
Let \(p_j=[u_j:v_j]\in\mathbb{PC}^1\) be the point corresponding to
\(\hat{z}_j\), and write
\[
\|(u,v)\|:=\bigl(|u|^2+|v|^2\bigr)^{1/2}.
\]
The projective chordal distance on \(\mathbb{PC}^1\) is
\[
[p_i,p_j]
:=
\frac{|u_i v_j-u_j v_i|}
{\|(u_i,v_i)\|\,\|(u_j,v_j)\|}.
\]
Under the identification \(\mathbb{PC}^1\simeq\Sphere^2\), this satisfies
\begin{equation}\label{eq:projective-chordal-distance}
\|\hat{z}_i-\hat{z}_j\|
=2[p_i,p_j]
=\frac{2|u_i v_j-u_j v_i|}
{\|(u_i,v_i)\|\,\|(u_j,v_j)\|}.
\end{equation}
Therefore maximizing the product of Euclidean chordal distances on
\(\Sphere^2\) is equivalent to maximizing
\[
\prod_{1\le i<j\le d}[p_i,p_j].
\]
Now set \(k=d-1\). Consider the monomial basis of
\(H^0(\mathbb{PC}^1,\mathcal O(k))\),
\[
S_\ell(U,V)=U^{k-\ell}V^\ell,
\qquad
\ell=0,\ldots,k.
\]
For \(d=k+1\) points \(p_1,\ldots,p_d\), the homogeneous Vandermonde determinant is,
up to sign,
\[
\det\bigl(S_\ell(u_j,v_j)\bigr)_{\ell=0,\ldots,k;\ j=1,\ldots,d}
=
\prod_{1\le i<j\le d}(u_i v_j-u_j v_i).
\]
The Fubini--Study metric on \(\mathcal O(k)\) gives
\[
\left|
\det\bigl(S_\ell(p_j)\bigr)
\right|_{k\phi_{\rm FS}}
=
\frac{
\displaystyle
\prod_{1\le i<j\le d}|u_i v_j-u_j v_i|
}{
\displaystyle
\prod_{j=1}^d \|(u_j,v_j)\|^k
}.
\]
Since \(k=d-1\), each factor \(\|(u_j,v_j)\|\) occurs exactly \(d-1\) times
in
\[
\prod_{1\le i<j\le d}
\|(u_i,v_i)\|\,\|(u_j,v_j)\|.
\]
Hence
\[
\left|
\det\bigl(S_\ell(p_j)\bigr)
\right|_{k\phi_{\rm FS}}
=
\prod_{1\le i<j\le d}[p_i,p_j].
\]
Thus the maximizers of the Fubini--Study weighted Vandermonde determinant for
\(\mathcal O(k)\) are exactly the elliptic Fekete configurations, up to the
irrelevant constant factor \(2^{d(d-1)/2}\) coming from
\(\|\hat{z}_i-\hat{z}_j\|=2[p_i,p_j]\).
\end{proof}

\begin{lemma}\label{lem:fekete-node-is-max}
Let $\omega_d=\{\hat{z}_1,\ldots,\hat{z}_d\}\subset\Sphere^2$ be an elliptic Fekete configuration. Then, for every $1\le i\le d$,
\[
\mathcal Q(\omega_d\setminus\{\hat{z}_i\})
=
H_{\omega_d\setminus\{\hat{z}_i\}}(\hat{z}_i)
=
P_i.
\]
\end{lemma}

\begin{proof}
\leavevmode\phantomsection
\label{proof:fekete-node-is-max}
Fix $p\in\Sphere^2$. Replacing $\hat{z}_i$ by $p$ in the Fekete configuration and using the extremality of $\omega_d$, we get
\[
H_{\omega_d\setminus\{\hat{z}_i\}}(p)
\le
H_{\omega_d\setminus\{\hat{z}_i\}}(\hat{z}_i)
=
P_i.
\]
Taking the supremum over $p\in\Sphere^2$ gives the claim.
\end{proof}

\begin{lemma}\label{lem:product-jensen}
For any configuration of $d$ spherical points, with repetitions allowed,
\[
\int_{\Sphere^2}\log H_{\omega_d}(p)\,\mathrm{d}\sigma(p)=-\kappa d,
\qquad
I_d^{1/2}\ge e^{-\kappa d},
\qquad
\mathcal Q(\omega_d)\ge e^{-\kappa d}.
\]
\end{lemma}
\begin{proof}
\leavevmode\phantomsection
\label{proof:product-jensen}
The logarithmic singularity is integrable. By \cref{def:kappa}, each term $\log\|p-\hat{z}_j\|$ has integral $-\kappa$. Summing proves the identity. Jensen's inequality gives
\[
\log I_d\ge\int_{\Sphere^2}\log H_{\omega_d}(p)^2\,\mathrm{d}\sigma(p)=-2\kappa d.
\]
Finally, since $H_{\omega_d}(p)\le\mathcal Q(\omega_d)$
and $\sigma(\Sphere^2)=1$,
\[
I_d=\int_{\Sphere^2}H_{\omega_d}(p)^2\,\mathrm{d}\sigma(p)
\le\mathcal Q(\omega_d)^2.
\]
Thus $\mathcal Q(\omega_d)\ge I_d^{1/2}\ge e^{-\kappa d}$.
\end{proof}

\begin{lemma}\label{lem:spherical-nikolskii}
For any configuration $\omega_d$ of $d\ge1$ spherical points,
possibly repeated,
\[
\mathcal Q(\omega_d)\le\sqrt{d+1}\,I_d^{1/2}.
\]
The constant $\sqrt{d+1}$ is optimal, and equality holds
if and only if all the points coincide.
\end{lemma}

\begin{proof}
\leavevmode\phantomsection
\label{proof:spherical-nikolskii}
Let
\[
F(U,V):=\prod_{j=1}^d(u_jV-v_jU),
\]
where $[u_j:v_j]$ represents $\hat z_j$.
By \cref{eq:projective-chordal-distance}, $H_{\omega_d}$
is a positive constant multiple of $|F|_{d\phi_{\rm FS}}$.
It therefore suffices to compare the supremum and $L^2$
norms of this section.

Choose a point $p_*$ where $|F|_{d\phi_{\rm FS}}$ attains
its maximum, and choose $A\in\mathrm{SU}(2)$ with
$Ap_*=[1:0]$. Apply the induced rotation to the entire
configuration and replace $F$ by $F\circ A^{-1}$.
By \cref{eq:fubini-study-unitary-invariance} and invariance
of $\sigma$, both norms are unchanged.
Denote the transformed section again by $F$.

In the affine coordinate $z=v/u$, write
\[
g(z):=F(1,z)=\sum_{k=0}^d a_kz^k.
\]
The normalized spherical measure is
$\mathrm{d}A(z)/(\pi(1+|z|^2)^2)$, where $\mathrm{d}A$ is planar area.
Writing $z=re^{i\theta}$, integration over $\theta$
eliminates the mixed terms in $|g(z)|^2$.
The substitution $u=r^2$ and evaluation of the resulting
beta integrals give
\begin{align*}
\int_{\Sphere^2}|F|_{d\phi_{\rm FS}}^2\,\mathrm{d}\sigma
&=\frac1\pi\int_{\C}
\frac{|g(z)|^2}{(1+|z|^2)^{d+2}}\,\mathrm{d}A(z)\\
&=\sum_{k=0}^d |a_k|^2
\int_0^\infty\frac{u^k}{(1+u)^{d+2}}\,\mathrm{d}u\\
&=\frac1{d+1}\sum_{k=0}^d
\frac{|a_k|^2}{\binom dk}
\ge\frac{|a_0|^2}{d+1}.
\end{align*}
The same computation gives
\begin{equation}\label{eq:bombieri-spherical-norm}
\int_{\Sphere^2}|F|_{d\phi_{\rm FS}}^2\,\mathrm{d}\sigma
=\frac{\|F\|_{\mathrm{BW}}^2}{d+1}.
\end{equation}
Since $[1:0]$ is a maximizing point,
\[
\max_{p\in\Sphere^2}|F(p)|_{d\phi_{\rm FS}}=|a_0|.
\]
The preceding inequality proves the desired estimate.

In these same coordinates, equality holds if and only if
$a_k=0$ for every $k\ge1$, so that $F(U,V)=a_0U^d$.
This is equivalent to all roots coinciding at $[0:1]$,
the antipode of the maximizing point.
Undoing the rotation shows that equality holds precisely
when all spherical points coincide.
The constant remains optimal for distinct points, since
such configurations can approach a configuration with all
points coincident.
\end{proof}

\begin{lemma}\label{lem:fekete-lagrange-local-stability}
There exist universal constants $a,b>0$ such that, for every elliptic Fekete configuration $\omega_d=\{\hat{z}_1,\ldots,\hat{z}_d\}$ with $d\ge2$,
\[
H_{\omega_d\setminus\{\hat{z}_j\}}(p)\ge aP_j
\qquad\text{whenever}\qquad
\|p-\hat{z}_j\|\le \frac{b}{\sqrt d}.
\]
\end{lemma}

\begin{proof}
\leavevmode\phantomsection
\label{proof:fekete-lagrange-local-stability}
Fix $j$ and set $n=d-1$. Since rotations preserve chordal
distances, we may rotate the entire configuration and the
evaluation point so that $\hat z_j$ corresponds to $[1:0]$.
We keep the same notation for the rotated configuration.

Let $p(z)\in\Sphere^2$ correspond to $[1:z]$, so that
$p(0)=\hat z_j$. For $k\ne j$, choose homogeneous
representatives $[u_k:v_k]$ of $\hat z_k$.
The nodes are distinct, so $v_k\ne0$. Define
\[
g(z):=\prod_{k\ne j}\left(1-\frac{u_k}{v_k}z\right).
\]
This polynomial has degree at most $n$ and satisfies $g(0)=1$.
By \cref{eq:projective-chordal-distance,lem:fekete-node-is-max},
\[
\frac{H_{\omega_d\setminus\{\hat z_j\}}(p(z))}{P_j}
=
\frac{|g(z)|}{(1+|z|^2)^{n/2}}
\le1,
\qquad z\in\C.
\]
To obtain bounds independent of $n$, put
$h(w):=g(w/\sqrt n)$. For $|w|\le1$,
\[
|h(w)|
\le\left(1+\frac{|w|^2}{n}\right)^{n/2}
\le e^{1/2}.
\]
Write $h(w)=1+\sum_{m=1}^n c_mw^m$.
By Cauchy's integral formula, for $1\le m\le n$,
\[
c_m=\frac{1}{2\pi i}\int_{|\zeta|=1}
\frac{h(\zeta)}{\zeta^{m+1}}\,\mathrm{d}\zeta.
\]
Since $|\zeta|=1$ and $|h(\zeta)|\le e^{1/2}$ on this circle,
\[
|c_m|
\le\frac1{2\pi}\int_{|\zeta|=1}|h(\zeta)|\,|\mathrm{d}\zeta|
\le e^{1/2}.
\]
Hence, for $|w|\le1/4$,
\[
|h(w)-1|
\le e^{1/2}\sum_{m=1}^{\infty}4^{-m}
=\frac{e^{1/2}}3<\frac23,
\]
and therefore $|h(w)|\ge1/3$.

Returning to the variable $z$, whenever
$|z|\le(4\sqrt n)^{-1}$ we have $|g(z)|\ge1/3$ and
\[
(1+|z|^2)^{n/2}
\le e^{n|z|^2/2}
\le e^{1/32}.
\]
Consequently,
\[
\frac{H_{\omega_d\setminus\{\hat z_j\}}(p(z))}{P_j}
\ge\frac13e^{-1/32}.
\]
Finally,
\[
\|p(z)-p(0)\|=\frac{2|z|}{\sqrt{1+|z|^2}}.
\]
Thus the disk $|z|\le(4\sqrt n)^{-1}$ corresponds to the
closed chordal cap centered at $\hat z_j$ with radius
\[
\frac2{\sqrt{16n+1}}
=\frac2{\sqrt{16d-15}}
\ge\frac1{2\sqrt d}.
\]
Undoing the rotation proves the lemma with
$a=\frac13e^{-1/32}$ and $b=\frac12$, for every $d\ge2$.
\end{proof}

\begin{lemma}\label{lem:fekete-pairwise-products}
There are universal $c,C>0$ such that, for every elliptic Fekete configuration with $d\ge2$ and every $i\ne j$,
\begin{equation}\label{eq:fekete-pairwise-products}
\frac{c}{\sqrt d\,\|\hat{z}_i-\hat{z}_j\|}
\le\frac{P_j}{P_i}
\le C\sqrt d\,\|\hat{z}_i-\hat{z}_j\|.
\end{equation}
\end{lemma}

\begin{proof}
\leavevmode\phantomsection
\label{proof:fekete-pairwise-products}
Let $a,b>0$ be the constants in
\cref{lem:fekete-lagrange-local-stability}.
Since
$H_{\omega_d\setminus\{\hat z_j\}}(\hat z_i)=0$
for $i\ne j$, that lemma implies
$\delta(\omega_d)>b/\sqrt d$.
Fix $0<\eta<\min\{b/2,1\}$ and choose $p\in\Sphere^2$
with $\|p-\hat z_i\|=\eta/\sqrt d$.
Then $p\ne\hat z_j$ and
$\eta/\sqrt d<\|\hat z_i-\hat z_j\|/2$.
Hence
\begin{align*}
P_j
&\ge H_{\omega_d\setminus\{\hat{z}_j\}}(p)\\
&=H_{\omega_d\setminus\{\hat{z}_i\}}(p)
\frac{\|p-\hat{z}_i\|}{\|p-\hat{z}_j\|}\\
&\ge aP_i\frac{\eta/\sqrt d}{\|\hat{z}_i-\hat{z}_j\|+\eta/\sqrt d}
\ge\frac{2a\eta}{3}\frac{P_i}{\sqrt d\,\|\hat{z}_i-\hat{z}_j\|}.
\end{align*}
The first inequality is \cref{lem:fekete-node-is-max}. Interchanging $i$ and $j$ gives the other side of \cref{eq:fekete-pairwise-products}.
\end{proof}

\subsection{Separation and Riesz estimates}\label{sec:proofs:geometry}

\begin{proof}[Proof of \cref{th:local-minimum-inverse-fourth}]
\leavevmode\phantomsection
\label{proof:local-minimum-inverse-fourth}
Set $n=d-1$. For each node $\hat z_i$, choose a unit
representative $v_i\in\C^2$ of its projective point and a vector
$w_i$ such that $(v_i,w_i)$ is an orthonormal basis with
$\det(v_i,w_i)=1$. 
We describe nearby configurations using the independent local
coordinates $[v_i+\zeta_iw_i]$, where
$\zeta=(\zeta_1,\ldots,\zeta_d)\in\C^d$.
Put
\[
G(\zeta):=\prod_{i<j}
\det(v_i+\zeta_iw_i,v_j+\zeta_jw_j).
\]
Since $\|v_i+\zeta_iw_i\|^2=1+|\zeta_i|^2$ and each node
occurs in $n$ unordered pairs,
\cref{eq:projective-chordal-distance} gives the product
of pairwise chordal distances in the perturbed configuration as
\[
2^{\binom d2}
\frac{|G(\zeta)|}
{\displaystyle\prod_{i=1}^d(1+|\zeta_i|^2)^{n/2}}.
\]
The logarithmic energy is minus twice the logarithm of this
product. Therefore, local minimality implies that
\[
\Phi(\zeta):=\log|G(\zeta)|
-\frac n2\sum_{i=1}^d\log(1+|\zeta_i|^2)
\]
has a local maximum at $\zeta=0$.

The nodes are distinct, hence $G(0)\ne0$, and we may choose a holomorphic logarithm $L=\log G$ near $0$.
Let $A=(A_{ij})_{i,j=1}^d$ be the complex symmetric matrix
\[
A_{ij}:=\partial_{\zeta_i}\partial_{\zeta_j}L(0).
\]
For each fixed $\zeta\in\C^d$, the real-valued function
$t\mapsto\Phi(t\zeta)$ has a local maximum at $t=0$.
Its first derivative therefore vanishes, and Taylor's
formula gives
\[
\Phi(t\zeta)-\Phi(0)
=\frac{t^2}{2}
\left(\operatorname{Re}(\zeta^TA\zeta)-n\|\zeta\|^2\right)
+o(t^2),
\qquad t\to0,\quad t\in\R.
\]
Here $T$ denotes ordinary transpose, while $\|\cdot\|$ denotes the standard Hermitian norm on $\mathbb C^d$.
Since the left-hand side is nonpositive for sufficiently
small $t$, dividing by $t^2$ and letting $t\to0$ yields
\[
\operatorname{Re}(\zeta^TA\zeta)-n\|\zeta\|^2\le0.
\]
Applying this inequality to $e^{i\theta}\zeta$ and choosing
$\theta$ so that $e^{2i\theta}\zeta^TA\zeta$ is real and
nonnegative gives
\[
|\zeta^TA\zeta|\le n\|\zeta\|^2.
\]
For unit vectors $u,v\in\C^d$, symmetry of $A$ gives
\[
4u^TAv=(u+v)^TA(u+v)-(u-v)^TA(u-v).
\]
Consequently,
\[
|u^TAv|\le\frac n4
\bigl(\|u+v\|^2+\|u-v\|^2\bigr)=n.
\]
Let $e_i$ be the $i$-th standard basis vector. If $Ae_i\ne0$,
choose
\[
v=e_i,
\qquad
u=\frac{\overline{Ae_i}}{\|Ae_i\|}.
\]
Then $u$ and $v$ are unit vectors and
\[
u^T A e_i
=
\frac{(\overline{Ae_i})^T Ae_i}{\|Ae_i\|}
=
\|Ae_i\|.
\]
Hence $\|Ae_i\|\le n$.
The same bound is immediate when $Ae_i=0$. 
Hence every column satisfies
\[
\sum_{j=1}^d|A_{ji}|^2=\|Ae_i\|^2\le n^2.
\]
It remains to compute the off-diagonal entries. Fix $i<j$.
Only the factor
\[
D(\zeta_i,\zeta_j):=
\det(v_i+\zeta_iw_i,v_j+\zeta_jw_j)
\]
contributes to $A_{ij}$. In the basis $(v_i,w_i)$, write
\[
v_j=\alpha v_i+\beta w_i,
\qquad
w_j=-\overline\beta\,v_i+\overline\alpha\,w_i,
\qquad |\alpha|^2+|\beta|^2=1.
\]
The expression for $w_j$ follows from orthonormality and
$\det(v_j,w_j)=1$. Expanding the determinant gives
\[
D(\zeta_i,\zeta_j)
=\beta-\alpha\zeta_i+\overline\alpha\,\zeta_j
+\overline\beta\,\zeta_i\zeta_j,
\qquad \beta=\det(v_i,v_j)\ne0.
\]
Differentiating its logarithm at the origin, we obtain
\[
A_{ij}
=\left.
\frac{D\,\partial_{\zeta_i}\partial_{\zeta_j}D
-(\partial_{\zeta_i}D)(\partial_{\zeta_j}D)}{D^2}
\right|_{(0,0)}
=\frac{|\beta|^2+|\alpha|^2}{\beta^2}
=\frac1{\det(v_i,v_j)^2}.
\]
By symmetry this holds for every $i\ne j$.
Since the representatives are unit vectors,
\cref{eq:projective-chordal-distance} now yields
\[
|A_{ij}|=\frac4{\|\hat z_i-\hat z_j\|^2},
\qquad i\ne j.
\]
The diagonal term $|A_{ii}|^2$ is nonnegative, so omitting it
from the squared norm of the $i$-th column gives
\[
16U_{4,i}(\omega_d)
=\sum_{j\ne i}|A_{ji}|^2
\le\sum_{j=1}^d|A_{ji}|^2
\le n^2=(d-1)^2.
\]
This proves \cref{eq:local-minimum-inverse-fourth}.

Finally, choose a pair $i,j$ attaining $\delta(\omega_d)$.
The other $d-2$ distances from $\hat z_i$ are at most $2$,
so each contributes at least $1/16$ to $U_{4,i}(\omega_d)$.
Therefore,
\[
\frac1{\delta(\omega_d)^4}+\frac{d-2}{16}
\le U_{4,i}(\omega_d)\le\frac{(d-1)^2}{16}.
\]
Rearranging gives
\[
\delta(\omega_d)\ge\frac2{(d^2-3d+3)^{1/4}},
\]
which is \cref{eq:local-minimum-separation}.
\end{proof}

\begin{lemma}\label{lem:line-bundle-fekete-local-count}
There exist universal constants $A,C>0$ with the following
property. Let $d\ge2$, set $k=d-1$, and let
$\omega_d\subset\Sphere^2$ be an elliptic Fekete configuration.
For every closed geodesic ball $B_{\rm geo}(x,r)$ in the unit
sphere $\Sphere^2$, with $x\in\Sphere^2$, $0<r\le\pi$, and
$r\sqrt k\ge A$,
\[
\left|
\frac{\#(\omega_d\cap B_{\rm geo}(x,r))}{d}
-\sigma(B_{\rm geo}(x,r))
\right|
\le
C\,\frac{\sigma(B_{\rm geo}(x,r))}{r\sqrt k}.
\]
\end{lemma}

\begin{proof}
\leavevmode\phantomsection
\label{proof:line-bundle-fekete-local-count}
Set $k=d-1$, and let $\mathcal F_k\subset\mathbb{CP}^1$
be the configuration corresponding to $\omega_d$ under
$\mathbb{CP}^1\simeq\Sphere^2$.
By \cref{lem:elliptic-fekete-as-line-bundle-fekete},
$\mathcal F_k$ is a Fekete configuration for
$(\mathcal O(k),k\phi_{\mathrm{FS}})$.
The auxiliary Hermitian metric in
\cite[Theorem~1]{LevOrtegaCerda2016} may be chosen arbitrarily.
We choose the pullback of the round metric on $\Sphere^2$,
so that its open metric balls are precisely
$B_{\rm geo}^{\circ}(x,r)$.
Moreover, the normalized curvature measure
\[
\nu:=
\frac{i\partial\bar\partial\phi_{\mathrm{FS}}}
{\displaystyle\int_{\mathbb{CP}^1}
i\partial\bar\partial\phi_{\mathrm{FS}}}
\]
corresponds to $\sigma$: it is an $\mathrm{SU}(2)$-invariant
probability measure, and hence is the normalized spherical
area measure under the above identification.
The normalized counting estimate in
\cite[Theorem~1]{LevOrtegaCerda2016} therefore gives
\[
\frac{\#(\mathcal F_k\cap B_{\rm geo}^{\circ}(x,r))}
{\#\mathcal F_k}
=
\left(1+O\bigl((r\sqrt{k})^{-1}\bigr)\right)
\nu(B_{\rm geo}^{\circ}(x,r)).
\]
Consequently, there exist universal constants $A,C>0$ such that
\[
\left|
\frac{\#(\omega_d\cap B_{\rm geo}^{\circ}(x,r))}{d}
-\sigma(B_{\rm geo}^{\circ}(x,r))
\right|
\le
C\,\frac{\sigma(B_{\rm geo}^{\circ}(x,r))}{r\sqrt{k}}
\]
whenever $0<r<\pi$ and $r\sqrt{k}\ge A$.
The constants are independent of the center, the radius,
the degree, and the chosen Fekete configuration.
Any fixed lower bound on $k$ required by the asymptotic
estimate can be enforced by increasing $A$, since
$r\le\pi$ and $r\sqrt{k}\ge A$ imply $k\ge A^2/\pi^2$.

Let $0<r<\pi$ satisfy $r\sqrt{k}\ge A$. For $n\ge1$, set
\[
r_n:=r+\frac{\pi-r}{n+1},
\qquad
B_n:=B_{\rm geo}^{\circ}(x,r_n).
\]
Then $r<r_n<\pi$ for every $n\ge1$, and
$\lim_{n\to\infty}r_n=r$. Moreover,
\[
B_{n+1}\subset B_n,
\qquad
\bigcap_{n=1}^{\infty}B_n=B_{\rm geo}(x,r).
\]
Indeed, a point belongs to every $B_n$ if and only if its
geodesic distance from $x$ is at most $r$.
By continuity from above for the finite measures $\sigma$
and $\sum_{p\in\omega_d}\delta_p$, we obtain
\[
\lim_{n\to\infty}\sigma(B_n)
=\sigma(B_{\rm geo}(x,r)),
\qquad
\lim_{n\to\infty}\#(\omega_d\cap B_n)
=\#(\omega_d\cap B_{\rm geo}(x,r)).
\]
Since $r_n\sqrt{k}\ge r\sqrt{k}\ge A$, the open-ball estimate
applies to every $B_n$, with the same constant $C$:
\[
\left|
\frac{\#(\omega_d\cap B_n)}{d}-\sigma(B_n)
\right|
\le C\,\frac{\sigma(B_n)}{r_n\sqrt{k}}.
\]
Taking $n\to\infty$ in this inequality, and using $r>0$, gives
\[
\left|
\frac{\#(\omega_d\cap B_{\rm geo}(x,r))}{d}
-\sigma(B_{\rm geo}(x,r))
\right|
\le C\,\frac{\sigma(B_{\rm geo}(x,r))}{r\sqrt{k}}.
\]
For $r=\pi$, we have $B_{\rm geo}(x,\pi)=\Sphere^2$,
so both normalized quantities on the left are equal to $1$.
\end{proof}

\begin{proof}[Proof of \cref{thm:fekete-cap-discrepancy}]
\leavevmode\phantomsection
\label{proof:fekete-cap-discrepancy}
By \cref{lem:elliptic-fekete-as-line-bundle-fekete}, we can apply \cref{lem:line-bundle-fekete-local-count} with $k=d-1$. Since $\sigma(B_{\rm geo}(q,r))=(1-\cos r)/2\le r^2/4$, for $r\sqrt k\ge A$ its conclusion gives
\[
\left|\#(\omega_d\cap B_{\rm geo}(q,r))-d\,\sigma(B_{\rm geo}(q,r))\right|
\le C\frac{d r}{\sqrt k}\le C' r\sqrt d.
\]
Suppose next that $r\sqrt k<A$ and $k$ is sufficiently large that $r_0:=A/\sqrt k<\pi$. Applying the same lemma at $r_0$ shows that the larger cap contains at most a constant depending only on $A$. Also $d\,\sigma(B_{\rm geo}(q,r))\le dA^2/(4k)=O(1)$. 
Since $r<r_0$, we have
$B_{\rm geo}(q,r)\subset B_{\rm geo}(q,r_0)$, and hence
\begin{align*}
\left|
\#(\omega_d\cap B_{\rm geo}(q,r))
-d\,\sigma(B_{\rm geo}(q,r))
\right|
&\le
\#(\omega_d\cap B_{\rm geo}(q,r))
+d\,\sigma(B_{\rm geo}(q,r))\\
&\le
\#(\omega_d\cap B_{\rm geo}(q,r_0))
+\frac{dA^2}{4k}
=O(1).
\end{align*}
The finitely many remaining degrees are covered by increasing the universal constant. This proves \cref{eq:fekete-local-count} for all radii, including $r=0$.

Finally, $B_{\rm ch}(q,t)=B_{\rm geo}(q,2\arcsin(t/2))$ and the geodesic radius is at most $\pi$. Dividing \cref{eq:fekete-local-count} by $d$ gives \cref{eq:fekete-cap-discrepancy}.
\end{proof}

\begin{proof}[Proof of \cref{th:covering-radius-fekete}]
\leavevmode\phantomsection
\label{proof:covering-radius-fekete}
The $d$ closed chordal caps of radius $\rho(\omega_d)$ centered at the nodes cover $\Sphere^2$. Their normalized areas are $\rho(\omega_d)^2/4$, so
\[
1\le\frac{d\,\rho(\omega_d)^2}{4}.
\]
This proves the lower bound for every configuration.

For the upper bound, take a fixed $A>0$ and $r=A/\sqrt d\le1$. Since $\sigma(B_{\rm geo}(q,r))\ge c r^2$ for $0\le r\le1$, \cref{eq:fekete-local-count} yields
\[
\#(\omega_d\cap B_{\rm geo}(q,r))\ge cA^2-C(1+A).
\]
Choose $A$ universally large enough to make the right-hand side positive. Every such cap contains a node, uniformly in its center. The chordal distance does not exceed the geodesic distance, so $\rho(\omega_d)\le A/\sqrt d$ for these degrees. The finitely many remaining degrees are absorbed into the constant using $\rho(\omega_d)\le2$.
\end{proof}

\begin{proof}[Proof of \cref{cor:fekete-discrepancy-distance-energy}]
\leavevmode\phantomsection
\label{proof:fekete-discrepancy-distance-energy}
By the definitions of the discrepancies and
$\sigma(\Sphere^2)=1$,
\[
D_2(\omega_d)^2
\le D_\infty(\omega_d)^2
\int_{\Sphere^2}\int_0^2 t\,\mathrm{d}t\,\mathrm{d}\sigma(q)
=2D_\infty(\omega_d)^2
\lesssim d^{-1}.
\]
With this normalization, Stolarsky's invariance principle
\cite{Stolarsky1973} reads 
\[
4D_2(\omega_d)^2
=\int_{\Sphere^2}\int_{\Sphere^2}\|p-q\|\,\mathrm{d}\sigma(p)\,\mathrm{d}\sigma(q)
-\frac{E_{-1}(\omega_d)}{d^2}.
\]
Rotational invariance gives the double integral as
\[
\frac12\int_{-1}^1\sqrt{2-2t}\,\mathrm{d}t=\frac43.
\]
Consequently $0\le\frac43d^2-E_{-1}(\omega_d)=4d^2D_2(\omega_d)^2\lesssim d$.
\end{proof}

\begin{proof}[Proof of \cref{cor:fekete-separation}]
\leavevmode\phantomsection
\label{proof:fekete-separation}
By \cref{th:covering-radius-fekete}, fix a universal $\beta\ge2$
such that $\rho(\omega_d)\le\beta/\sqrt d$. Suppose first that
$d>\beta^2$, set $R=\beta/\sqrt d<1$, and choose a pair
$\hat z_i,\hat z_j$ attaining $\delta(\omega_d)$.
The cap $B_{\rm ch}(\hat z_i,2R)$ has normalized area $R^2$,
whereas the union of the radius-$R$ caps about $\hat z_i$ and
$\hat z_j$ has area at most $R^2/2$. There is therefore a point
$p$ in the former cap and outside the latter two. 
By covering, $p$ lies within distance $R$ of a node $\hat z_k$ with $k\ne i,j$, and $\|\hat z_i-\hat z_k\|\le3R$.
This, together with \cref{eq:local-minimum-inverse-fourth}, give us
\[
\frac{1}{\delta(\omega_d)^4}+\frac{d^2}{81\beta^4}
\le U_{4,i}(\omega_d)
\le\frac{(d-1)^2}{16}\le\frac{d^2}{16}.
\]
Hence, for $d>\beta^2$,
\[
\sqrt d\,\delta(\omega_d)\ge
\alpha_\beta:=2\left(1-\frac{16}{81\beta^4}\right)^{-1/4}>2.
\]
For the finitely many remaining degrees, \cref{eq:local-minimum-separation}
gives
\[
\sqrt d\,\delta(\omega_d)
\ge2\left(\frac{d^2}{d^2-3d+3}\right)^{1/4}>2.
\]
Taking the minimum of these finitely many bounds and
$\alpha_\beta$ proves the claim with a universal $\alpha>2$.
\end{proof}


\subsubsection{Transfer from logarithmic to Riesz energy}
\label{sec:proofs:energy-transfer}
\leavevmode
Brauchart \cite[Section~3.2]{Brauchart2026} relates logarithmic
and critical Riesz energies by partial summation. Here we combine
the logarithmic-energy bound with explicit control of the short-range
tail to obtain upper bounds as well.

\begin{lemma}\label{lem:log-to-riesz-transfer}
Let $A\in\R$, $B>0$, and let $\omega_d=\{\hat{z}_1,\ldots,\hat{z}_d\}\subset\Sphere^2$
consist of $d\ge2$ distinct points. Write
\[
x_{ij}:=\frac{1+\langle\hat z_i,\hat z_j\rangle}{2}
=1-\frac{\|\hat z_i-\hat z_j\|^2}{4}.
\]
Suppose that
\begin{equation}\label{eq:transfer-log-hypothesis}
\Energy(\omega_d)\le\kappa d^2-\frac12d\log d+Ad
\end{equation}
and
\begin{equation}\label{eq:transfer-tail-hypothesis}
T_d:=
\sum_{i\ne j}
\frac{x_{ij}^{d+1}}{\|\hat z_i-\hat z_j\|^2}
\le Bd^2.
\end{equation}
Then
\begin{align*}
E_s(\omega_d) & =\frac{2^{1-s}}{2-s}d^2
+O_{A,B,s}(d^{1+s/2}),\qquad 0<s<2,\\
E_2(\omega_d) & =\frac14d^2\log d+O_{A,B}(d^2).
\end{align*}
More precisely, if these hypotheses hold with the same $A,B$
for all sufficiently large degrees in a family
$\{\omega_d\}_{d\in\mathcal D}$, then
\begin{align}
\liminf_{d\to\infty}
R_s(\omega_d)
&\ge c_s^-,\notag\\
\limsup_{d\to\infty}
R_s(\omega_d)
&\le c_s^+(A,B),\qquad 0<s<2,
\notag\\
\liminf_{d\to\infty}
R_2(\omega_d)
&\ge\frac{\gamma-1}{4},\label{eq:transfer-critical-lower}\\
\limsup_{d\to\infty}
R_2(\omega_d)
&\le\frac{A+\gamma-\log2}{2}+B.
\label{eq:transfer-critical-upper}
\end{align}
The constants $c_s^-,c_s^+(A,B)$ are given in
\cref{eq:transfer-second-order-constants}.
\end{lemma}

\begin{proof}
\leavevmode\phantomsection
\label{proof:log-to-riesz-transfer}
We expand the logarithmic and Riesz kernels in powers of $x_{ij}$.
The positivity of the resulting moment discrepancies allows us to
control the truncated Riesz sums using
\cref{eq:transfer-log-hypothesis}, while
\cref{eq:transfer-tail-hypothesis} controls the Riesz tails.

For each integer $k\ge0$, set
\begin{equation}\label{eq:positive-kernel-moments}
K_k(p,q):=\left(\frac{1+\langle p,q\rangle}{2}\right)^k,
\qquad
B_k:=\sum_{i,j=1}^d x_{ij}^k-\frac{d^2}{k+1}.
\end{equation}
We use the convention $x^0=1$, including at $x=0$, so that $B_0=0$.
For each fixed $p\in\Sphere^2$, rotational invariance gives
\[
\int_{\Sphere^2}K_k(p,q)\,\mathrm{d}\sigma(q)
=\frac12\int_{-1}^1
\left(\frac{1+t}{2}\right)^k\mathrm{d}t
=\frac1{k+1}.
\]
Hence, for the signed measure
$\nu_d:=\sum_{i=1}^d\delta_{\hat z_i}-d\sigma$,
\[
\iint K_k(p,q)\,\mathrm{d}\nu_d(p)\,\mathrm{d}\nu_d(q)
=\sum_{i,j=1}^d x_{ij}^k-\frac{2d^2}{k+1}
+\frac{d^2}{k+1}
=B_k.
\]
Writing $v(p):=(1,p)/\sqrt2\in\R^4$, we also have
\[
K_k(p,q)
=\langle v(p)^{\otimes k},v(q)^{\otimes k}\rangle.
\]
It follows that
\[
B_k
=\left\|\int v(p)^{\otimes k}\,\mathrm{d}\nu_d(p)\right\|^2
\ge0.
\]
Since $x_{ii}=1$, removing the diagonal terms yields,
for every $k\ge0$,
\[
\sum_{i\ne j}x_{ij}^k
=B_k+\frac{d^2}{k+1}-d.
\]
For an integer $m\ge1$, let
$\mathfrak h_m:=\sum_{k=1}^m1/k$
and define the logarithmic and inverse-square tails
\begin{equation}\label{eq:transfer-tails}
L_m:=\sum_{i\ne j}\sum_{k>m}\frac{x_{ij}^k}{k},
\qquad
T_m:=\frac14\sum_{i\ne j}\frac{x_{ij}^{m+1}}{1-x_{ij}}.
\end{equation}
Both tails are nonnegative, and $T_d$ agrees with the quantity in
\cref{eq:transfer-tail-hypothesis} because
$\|\hat z_i-\hat z_j\|^2=4(1-x_{ij})$.
Using $x_{ij}<1$ for $i\ne j$ and the series for $-\log(1-x)$,
we obtain
\begin{align*}
2\Energy(\omega_d)+2d(d-1)\log2
&=\sum_{i\ne j}-\log(1-x_{ij})\\
&=\sum_{k=1}^m\frac1k
\left(B_k+\frac{d^2}{k+1}-d\right)+L_m.
\end{align*}
Since
\[
\sum_{k=1}^m\frac1{k(k+1)}
=\sum_{k=1}^m\left(\frac1k-\frac1{k+1}\right)
=1-\frac1{m+1},
\]
rearranging gives the exact identity
\begin{equation}\label{eq:logarithmic-truncation-identity}
\sum_{k=1}^m\frac{B_k}{k}+L_m
=2\Energy(\omega_d)+2d(d-1)\log2
+d\mathfrak h_m-d^2+\frac{d^2}{m+1}.
\end{equation}
Using $2\kappa=1-2\log2$ and
\cref{eq:transfer-log-hypothesis}, the right-hand side is
at most $d\eta_{d,m}(A)$, where
\[
\eta_{d,m}(A):=
2A+\mathfrak h_m-\log d-2\log2+\frac{d}{m+1}.
\]
Taking $m=d$ and discarding $L_d\ge0$ therefore gives
\begin{equation}\label{eq:positive-kernel-moment-bound}
0\le\sum_{k=1}^d\frac{B_k}{k}
\le d\eta_{d,d}(A)
=\bigl(\eta(A)+o(1)\bigr)d.
\end{equation}
Here we used $\mathfrak h_d=\log d+\gamma+O(d^{-1})$ and
$\eta(A)=2A+\gamma+1-2\log2$.
This is the moment bound used below to transfer the logarithmic
estimate to Riesz energies. It requires neither $A>0$ nor the
tail hypothesis.

We first treat the critical case $s=2$.
The finite geometric series gives
\[
\frac1{4(1-x)}
=\frac14\sum_{k=0}^m x^k
+\frac{x^{m+1}}{4(1-x)},
\qquad 0\le x<1.
\]
Summing over $i\ne j$ and removing the diagonal terms as above
yields
\begin{equation}\label{eq:critical-energy-truncation}
E_2(\omega_d)
=\frac{d^2}{4}\mathfrak h_{m+1}
-\frac{d(m+1)}4
+\frac14\sum_{k=1}^mB_k+T_m.
\end{equation}
For $m=d$, positivity and
\cref{eq:positive-kernel-moment-bound} imply
\[
0\le\sum_{k=1}^dB_k
=\sum_{k=1}^d k\frac{B_k}{k}
\le d\sum_{k=1}^d\frac{B_k}{k}
\le d^2\eta_{d,d}(A).
\]
Together with $0\le T_d\le Bd^2$, this gives
\[
\begin{aligned}
0
&\le E_2(\omega_d)
-\left(\frac{d^2}{4}\mathfrak h_{d+1}
-\frac{d(d+1)}4\right)\\
&\le\frac{d^2}{4}\eta_{d,d}(A)+Bd^2.
\end{aligned}
\]
Moreover,
\[
\frac{d^2}{4}\mathfrak h_{d+1}-\frac{d(d+1)}4
=\frac14d^2\log d
+\frac{\gamma-1}{4}d^2+o(d^2).
\]
These bounds prove the critical order estimate.
For a family satisfying the hypotheses with fixed $A,B$,
they also give
\[
\liminf_{d\to\infty}R_2(\omega_d)
\ge\frac{\gamma-1}{4},
\]
and
\[
\limsup_{d\to\infty}R_2(\omega_d)
\le\frac{\gamma-1}{4}+\frac{\eta(A)}4+B
=\frac{A+\gamma-\log2}{2}+B,
\]
which are
\cref{eq:transfer-critical-lower,eq:transfer-critical-upper}.

Now let $0<s<2$, put $a=s/2$, and write
\[
(1-x)^{-a}=\sum_{k=0}^\infty b_kx^k,
\qquad
b_k:=\frac{\Gamma(k+a)}{\Gamma(a)\Gamma(k+1)}.
\]
The coefficients are positive and satisfy
$(k+1)b_{k+1}=(k+a)b_k$.
Since $0<a<1$, this recurrence shows that $b_k$ is decreasing
and $kb_k$ is increasing for $k\ge1$.
For each fixed $a\in(0,1)$, Stirling's formula gives
\[
b_k=\frac{k^{a-1}}{\Gamma(a)}
\left(1+O_a(k^{-1})\right),
\qquad k\to\infty.
\]
The recurrence also yields
\[
\sum_{k=0}^d b_k=\frac{d+a}{a}\,b_d,
\qquad
\sum_{k>d}\frac{b_k}{k+1}=\frac{b_{d+1}}{1-a}.
\]
The first identity follows by induction; the second follows
by summing
$b_k-b_{k+1}=(1-a)b_k/(k+1)$ and using $b_k\to0$.
Consequently, as $d\to\infty$ with $a$ fixed,
\begin{equation}\label{eq:binomial-coefficient-asymptotics}
\begin{aligned}
\sum_{k=0}^d b_k
&=\frac{d^a}{\Gamma(a)a}
\left(1+O_a(d^{-1})\right),\\
\sum_{k>d}\frac{b_k}{k+1}
&=\frac{d^{a-1}}{\Gamma(a)(1-a)}
\left(1+O_a(d^{-1})\right),\\
db_d
&=\frac{d^a}{\Gamma(a)}
\left(1+O_a(d^{-1})\right),\\
b_{d+1}
&=\frac{d^{a-1}}{\Gamma(a)}
\left(1+O_a(d^{-1})\right).
\end{aligned}
\end{equation}
Termwise integration, justified by nonnegativity, gives
\[
\sum_{k=0}^\infty\frac{b_k}{k+1}
=\int_0^1(1-x)^{-a}\,\mathrm{d}x
=\frac1{1-a}.
\]
Thus
$2^{-s}\sum_{k\ge0}b_k/(k+1)=2^{1-s}/(2-s)$.

Define the remainder
\[
R_d:=\sum_{i\ne j}\sum_{k>d}b_kx_{ij}^k.
\]
Since $b_k\le b_{d+1}$ for $k>d$, the geometric series and
\cref{eq:transfer-tail-hypothesis} imply
\[
0\le R_d
\le b_{d+1}\sum_{i\ne j}\frac{x_{ij}^{d+1}}{1-x_{ij}}
=4b_{d+1}T_d
\le4Bb_{d+1}d^2.
\]
Using
$\|\hat z_i-\hat z_j\|^{-s}=2^{-s}(1-x_{ij})^{-a}$
and removing the diagonal terms, we obtain
\[
E_s(\omega_d)
=2^{-s}\left(
\sum_{k=0}^d b_k
\left(B_k+\frac{d^2}{k+1}-d\right)+R_d
\right).
\]
Subtracting the full continuous contribution
$2^{-s}d^2\sum_{k\ge0}b_k/(k+1)$ and using $B_0=0$ gives
\begin{equation}\label{eq:subcritical-energy-truncation}
\begin{aligned}
E_s(\omega_d)-\frac{2^{1-s}}{2-s}d^2
=2^{-s}\Bigl(&\sum_{k=1}^d b_kB_k
-d\sum_{k=0}^d b_k\\
&-d^2\sum_{k>d}\frac{b_k}{k+1}+R_d\Bigr).
\end{aligned}
\end{equation}
The first and last terms in parentheses are nonnegative.
To bound the first, use the monotonicity of $kb_k$ and
\cref{eq:positive-kernel-moment-bound}:
\[
0\le\sum_{k=1}^d b_kB_k
=\sum_{k=1}^d(kb_k)\frac{B_k}{k}
\le db_d\sum_{k=1}^d\frac{B_k}{k}
\le d^2b_d\eta_{d,d}(A).
\]
Together with the bound for $R_d$ and
\cref{eq:binomial-coefficient-asymptotics}, this shows that
each of the four terms on the right-hand side of
\cref{eq:subcritical-energy-truncation} is
$O_{A,B,s}(d^{1+a})$.
Since $a=s/2$, the subcritical order estimate follows.

For the explicit bounds along a family with fixed $A,B$,
note that
\[
\begin{aligned}
\frac{
d\sum_{k=0}^d b_k
+d^2\sum_{k>d}b_k/(k+1)
}{d^{1+a}}
&\longrightarrow
\frac1{\Gamma(a)}
\left(\frac1a+\frac1{1-a}\right)\\
&=\frac1{\Gamma(a)a(1-a)},
\end{aligned}
\]
whereas the two nonnegative terms satisfy
\[
0\le
\frac{\sum_{k=1}^d b_kB_k+R_d}{d^{1+a}}
\le\frac{\eta(A)+4B}{\Gamma(a)}+o(1).
\]
Dividing \cref{eq:subcritical-energy-truncation}
by $d^{1+a}$ therefore yields
\[
\liminf_{d\to\infty}R_s(\omega_d)
\ge-\frac{2^{-s}}{\Gamma(a)a(1-a)}
=c_s^-,
\]
and
\[
\limsup_{d\to\infty}R_s(\omega_d)
\le c_s^-+\frac{2^{-s}}{\Gamma(a)}
\bigl(\eta(A)+4B\bigr)
=c_s^+(A,B).
\]
\end{proof}

Since $\eta(A)=2A+\gamma+1-2\log2$, decreasing $A$ improves the constant in \cref{eq:positive-kernel-moment-bound}, without changing the order of the bound in $d$.
In particular, $\eta(C_{\rm conj})=0.079710\ldots>0$.
Even at $A=C_{\rm conj}$, the upper bound in
\cref{eq:transfer-critical-upper} is $C_{2,\rm conj}+B$, not $C_{2,\rm conj}$.
For local logarithmic minimizers, the next lemma improves the critical estimate by retaining both tails together.

\begin{lemma}\label{lem:local-minimum-critical-transfer}
For every local logarithmic minimizer $\omega_d\subset\Sphere^2$,
$d\ge2$, the tail in \cref{eq:transfer-tail-hypothesis} satisfies
\begin{equation}\label{eq:inverse-fourth-tail-bound}
T_d\le\frac{4}{ed}E_4(\omega_d)
\le\frac{(d-1)^2}{4e}.
\end{equation}
If a family $\{\omega_d\}_{d\in\mathcal D}$ of such minimizers satisfies \cref{eq:transfer-log-hypothesis} with a fixed $A\in\R$ for all sufficiently large $d$, then
\begin{equation}\label{eq:critical-refined-transfer}
\limsup_{d\to\infty}
R_2(\omega_d)
\le\Phi_A(t),\qquad t\ge1,
\end{equation}
where $\Phi_A$ is defined in \cref{eq:critical-transfer-function}.
\end{lemma}

\begin{proof}
\leavevmode\phantomsection
\label{proof:local-minimum-critical-transfer}
Throughout this proof, write
$r_{ij}:=\|\hat z_i-\hat z_j\|$.
Summing the inverse-fourth-power bound in
\cref{eq:local-minimum-inverse-fourth} over the nodes,
and bounding each summand individually, gives
\begin{equation}\label{eq:local-minimum-fourth-energy-separation}
E_4(\omega_d)\le\frac{d(d-1)^2}{16},
\qquad
r_{ij}^2\ge\frac4{d-1}\quad(i\ne j).
\end{equation}
Since $x_{ij}=1-r_{ij}^2/4$, we have
$x_{ij}^{d+1}\le e^{-dr_{ij}^2/4}$.
Also,
\[
y e^{-dy/4}\le\frac4{ed},
\qquad y\ge0.
\]
Consequently,
\[
\frac{x_{ij}^{d+1}}{r_{ij}^2}
\le
\frac{r_{ij}^2e^{-dr_{ij}^2/4}}{r_{ij}^4}
\le\frac4{ed\,r_{ij}^4}.
\]
Summing over $i\ne j$, we obtain
\[
T_d\le\frac4{ed}E_4(\omega_d)
\le\frac{(d-1)^2}{4e},
\]
which proves \cref{eq:inverse-fourth-tail-bound}.

For the second assertion, consider a family satisfying
\cref{eq:transfer-log-hypothesis} with a fixed $A$,
and take $d$ sufficiently large that this bound holds.
Fix $t\ge1$ and set $m=\lfloor td\rfloor\ge d$.
Use the notation of
\cref{eq:positive-kernel-moments,eq:transfer-tails}.
By \cref{eq:logarithmic-truncation-identity} and the
logarithmic-energy hypothesis,
\[
\sum_{k=1}^m\frac{B_k}{k}+L_m
\le d\eta_{d,m}(A).
\]
Since $B_k\ge0$, it follows that
\[
\sum_{k=1}^mB_k
\le m\sum_{k=1}^m\frac{B_k}{k}
\le md\eta_{d,m}(A)-mL_m.
\]
Substituting this bound into
\cref{eq:critical-energy-truncation} gives
\begin{equation}\label{eq:critical-combined-tail-estimate}
E_2(\omega_d)\le
\frac{d^2}{4}\mathfrak h_{m+1}
-\frac{d(m+1)}4
+\frac{md}{4}\eta_{d,m}(A)
+T_m-\frac m4L_m.
\end{equation}
We retain the negative term $-mL_m/4$ and estimate it
together with $T_m$ to exploit the cancellation between
the two tails.

For $0\le x<1$, we have
\[
\frac{x^{m+1}}{1-x}
-m\sum_{k>m}\frac{x^k}{k}
=
\int_0^x\frac{v^m}{(1-v)^2}\,\mathrm{d}v.
\]
Indeed, both sides vanish at $x=0$ and have derivative
$x^m/(1-x)^2$.
Set $u=m(1-x)>0$.
The substitution $y=m(1-v)$ gives
\[
\begin{aligned}
\int_0^x\frac{v^m}{(1-v)^2}\,\mathrm{d}v
&=
m\int_u^m\frac{(1-y/m)^m}{y^2}\,\mathrm{d}y\\
&\le
m\int_u^\infty\frac{e^{-y}}{y^2}\,\mathrm{d}y
=\frac{m}{u^2}F(u),
\end{aligned}
\]
where
\[
F(u):=
u^2\int_u^\infty\frac{e^{-y}}{y^2}\,\mathrm{d}y,
\qquad u>0.
\]
Here we used $(1-y/m)^m\le e^{-y}$ for $0\le y\le m$,
and then enlarged the interval of integration.
This integral representation shows that $F(u)\ge0$.
For $u\ge1$, integration by parts yields
\[
F(u)=u e^{-u}-u^2J(u),
\qquad
J(u)=\int_u^\infty\frac{e^{-y}}{y}\,\mathrm{d}y,
\]
and hence
\[
F'(u)=e^{-u}-2uJ(u).
\]
To determine the sign of this derivative, write
\[
J(u)=e^{-u}\int_0^\infty
\frac{e^{-v}}{u+v}\,\mathrm{d}v.
\]
The measure $e^{-v}\,\mathrm{d}v$ on $[0,\infty)$ is a
probability measure with mean $1$.
Since $v\mapsto(u+v)^{-1}$ is convex, Jensen's inequality gives
\[
J(u)\ge\frac{e^{-u}}{u+1}.
\]
Therefore,
\[
F'(u)
\le e^{-u}\left(1-\frac{2u}{u+1}\right)
=e^{-u}\frac{1-u}{u+1}
\le0,
\qquad u\ge1.
\]
Thus $F$ is nonnegative and nonincreasing on $[1,\infty)$.
For each pair $i\ne j$, set
\[
u_{ij}:=m(1-x_{ij})=\frac{mr_{ij}^2}{4}.
\]
By \cref{eq:local-minimum-fourth-energy-separation},
\[
u_{ij}\ge\frac{m}{d-1}=:q_d>1.
\]
The identity and integral estimate above therefore imply
\begin{align*}
0\le T_m-\frac m4L_m
&=
\frac14\sum_{i\ne j}
\int_0^{x_{ij}}\frac{v^m}{(1-v)^2}\,\mathrm{d}v\\
&\le
\frac4m\sum_{i\ne j}\frac{F(u_{ij})}{r_{ij}^4}\\
&\le
\frac4mF(q_d)E_4(\omega_d)\\
&\le
\frac{d(d-1)^2}{4m}F(q_d).
\end{align*}
As $d\to\infty$, we have $m/d\to t$ and $q_d\to t$.
By continuity of $F$, it follows that
\begin{equation}\label{eq:critical-combined-tail-limit}
\limsup_{d\to\infty}
\frac{T_m-mL_m/4}{d^2}
\le\frac{F(t)}{4t}
=\frac14\bigl(e^{-t}-tJ(t)\bigr).
\end{equation}
Finally, the harmonic-number asymptotic and $m/d\to t$ give
\[
\mathfrak h_{m+1}-\log d
\longrightarrow\gamma+\log t,
\]
and
\[
\eta_{d,m}(A)
\longrightarrow
2A+\gamma+\log t-2\log2+\frac1t.
\]
Subtracting $\frac14d^2\log d$ from
\cref{eq:critical-combined-tail-estimate},
dividing by $d^2$, and using
\cref{eq:critical-combined-tail-limit}, we obtain
\[
\begin{aligned}
\limsup_{d\to\infty}R_2(\omega_d)
\le\frac14\Bigl[
&\gamma+\log t-t\\
&+t\left(
2A+\gamma+\log t-2\log2+\frac1t
\right)\\
&+e^{-t}-tJ(t)
\Bigr].
\end{aligned}
\]
Since $\eta(A)=2A+\gamma+1-2\log2$, the expression on
the right equals
\[
\frac14\Bigl[
\gamma+\log t+1
+t\bigl(\eta(A)+\log t-2\bigr)
+e^{-t}-tJ(t)
\Bigr]
=\Phi_A(t).
\]
This proves \cref{eq:critical-refined-transfer}.
\end{proof}

\subsubsection{Applications to Fekete configurations}

\begin{proof}[Proof of \cref{thm:fekete-riesz-potentials}]
\leavevmode\phantomsection
\label{proof:fekete-riesz-potentials}
Choose the separation constant $\alpha$ from \cref{cor:fekete-separation} and put $a_d=\alpha/\sqrt d$. For each node, define
\[
N_i(t):=\#\{j\ne i:\|\hat{z}_i-\hat{z}_j\|\le t\},\qquad0\le t\le2.
\]
Then $N_i(t)=0$ for $t<a_d$. A chordal cap of radius $t$ has normalized area $t^2/4$ and geodesic radius $2\arcsin(t/2)\le\pi t/2$. 
Since the cap is centered at $\hat z_i$,
\[
\#(\omega_d\cap B_{\rm ch}(\hat z_i,t))=1+N_i(t).
\]
Thus \cref{eq:fekete-local-count} gives
\[
\left|N_i(t)+1-\frac{d\,t^2}{4}\right|
\le C_1(1+t\sqrt d).
\]
Set $R_i(t):=N_i(t)-d\,t^2/4$. By the preceding estimate
and the triangle inequality,
\[
|R_i(t)|
\le C_1(1+t\sqrt d)+1
\le (C_1+1)(1+t\sqrt d).
\]
Thus, with $C=C_1+1$,
\begin{equation}\label{eq:radial-count-remainder}
N_i(t)=\frac{d\,t^2}{4}+R_i(t),
\qquad
|R_i(t)|\le C(1+t\sqrt d),
\end{equation}
uniformly in $i$, $d$, $t$, and the configuration.
For $s>0$ and $0<r\le2$, the identity
\[
r^{-s}=2^{-s}+s\int_r^2 t^{-s-1}\,\mathrm{d}t
\]
gives, by summing the finite collection of pairwise distances,
\begin{equation}\label{eq:riesz-radial-representation}
U_{s,i}(\omega_d)=(d-1)2^{-s}+s\int_{a_d}^2\frac{N_i(t)}{t^{s+1}}\,\mathrm{d}t.
\end{equation}
Suppose first that $0<s<2$. 
Substituting \cref{eq:radial-count-remainder} into
\cref{eq:riesz-radial-representation}, we obtain
\[
U_{s,i}(\omega_d)
=(d-1)2^{-s}
+\frac{sd}{4}\int_{a_d}^{2}t^{1-s}\,\mathrm{d}t
+s\int_{a_d}^{2}\frac{R_i(t)}{t^{s+1}}\,\mathrm{d}t.
\]
Since $0<s<2$, evaluating the first integral gives
\[
U_{s,i}(\omega_d)-\frac{2^{1-s}}{2-s}\,d
=
-2^{-s}
-\frac{sd}{4(2-s)}a_d^{2-s}
+s\int_{a_d}^{2}\frac{R_i(t)}{t^{s+1}}\,\mathrm{d}t.
\]
The lower-cutoff term is $O_s(d^{s/2})$, while
\[
\left|s\int_{a_d}^2\frac{R_i(t)}{t^{s+1}}\,\mathrm{d}t\right|
\le C_s\left(d^{s/2}+\sqrt d\int_{a_d}^2t^{-s}\,\mathrm{d}t\right).
\]
The last integral is $O_s(1)$ for $0<s<1$, $O(\log d)$ for $s=1$, and $O_s(d^{(s-1)/2})$ for $1<s<2$. These bounds prove \cref{eq:fekete-local-riesz-subcritical} with precisely the stated remainders.

At $s=2$, \cref{eq:riesz-radial-representation} instead yields
\[
U_{2,i}(\omega_d)
=\frac{d-1}{4}+\frac d2\log\frac2{a_d}
+2\int_{a_d}^2\frac{R_i(t)}{t^3}\,\mathrm{d}t.
\]
The error satisfies
\[
\left|2\int_{a_d}^2\frac{R_i(t)}{t^3}\,\mathrm{d}t\right|
\le2C\int_{a_d}^2(t^{-3}+\sqrt d\,t^{-2})\,\mathrm{d}t
\le C(\alpha^{-2}+2\alpha^{-1})d.
\]
Since $\frac d2\log(2/a_d)=\frac14d\log d+\frac d2\log(2/\alpha)$, this proves \cref{eq:fekete-local-riesz-two}. Every local estimate is uniform in the node and the configuration. Summing this critical estimate proves \cref{eq:fekete-riesz-two}.

For the sharper global subcritical remainder, apply
\cref{lem:log-to-riesz-transfer}. The expansion
\cref{eq:asymp_min_energy} gives
\cref{eq:transfer-log-hypothesis} with a universal $A$,
including the finitely many small degrees after increasing it.
The tail bound \cref{eq:inverse-fourth-tail-bound} gives
\cref{eq:transfer-tail-hypothesis} with $B=1/(4e)$.
This proves \cref{eq:fekete-riesz-subcritical}.
\end{proof}

\begin{proof}[Proof of \cref{cor:fekete-riesz-second-order}]
\leavevmode\phantomsection
\label{proof:fekete-riesz-second-order}
Fix $\varepsilon>0$. By
\cref{eq:asymp_min_energy,eq:Clog-bounds}, every Fekete
configuration satisfies \cref{eq:transfer-log-hypothesis}
with $A=C_{\rm conj}+\varepsilon$ for all sufficiently large
$d$. No sign for the $o(d)$ remainder is assumed.
By \cref{eq:inverse-fourth-tail-bound}, we can take $B=1/(4e)$.
Applying \cref{lem:log-to-riesz-transfer} and then letting
$\varepsilon\downarrow0$ gives
\cref{eq:fekete-second-order-lower,eq:fekete-second-order-upper,eq:fekete-critical-second-order-lower}.
For $a=s/2\in(0,1)$, the subcritical upper constant is negative because
\[
\eta(C_{\rm conj})+e^{-1}=0.447590\ldots
<4\le\frac1{a(1-a)}.
\]
Its sign also follows directly from $C_{\rm conj}<0$,
$\gamma<1$, $2\log2>1$, and $e^{-1}<1$.

For the critical upper bound, fix $t\ge1$ and apply
\cref{lem:local-minimum-critical-transfer} with the same $A$.
Letting $\varepsilon\downarrow0$ gives the bound
$\Phi_{C_{\rm conj}}(t)$. Taking the infimum over $t\ge1$
proves \cref{eq:fekete-critical-second-order-upper}.
In particular, $t=4/3$ yields the stated value of $c_{\rm up}$.
To check its sign without relying on the displayed decimal,
use $J(t)\ge e^{-t}/(t+1)$ and $C_{\rm conj}<0$:
\[
4c_{\rm up}\le
\gamma+\log\frac43+1
+\frac43\left(\gamma-2\log2+\log\frac43-1\right)
+\frac37e^{-4/3}<0.
\]
All estimates use only degree-dependent energy bounds and
universal local-minimum constants, so the remainders are
uniform over the choice of Fekete configuration.
\end{proof}

\begin{proof}[Proof of \cref{prop:minimal-riesz-upper-comparison}]
\leavevmode\phantomsection
\label{proof:minimal-riesz-upper-comparison}
For each $d$, choose an elliptic Fekete configuration $\omega_d$.
Since $\mathcal E_s(d)\le E_s(\omega_d)$,
\cref{cor:fekete-riesz-second-order} gives the first inequality in
\cref{eq:minimal-riesz-second-order-upper}.
It remains to compare the constants.

Put $a=s/2\in(0,1)$, $u=a(1-a)$, and
\[
K:=\eta(C_{\rm conj})+e^{-1}
=2C_{\rm conj}+\gamma+1-2\log2+e^{-1}.
\]
The gamma reflection formula gives
\begin{equation}\label{eq:difconstants}
C_{\rm TM}(s)-C_{\rm F}(s)
=\frac{2^{-s}}{\Gamma(a)}\bigl(\Psi(a)-K\bigr),
\qquad
\Psi(a):=\frac1u+
\frac{\pi a(1+a)}{\sin(\pi a)}\zeta(1-a).
\end{equation}
We prove that $\Psi(a)>1/2$ and $K<1/2$.

For $0<q<1$, Euler--Maclaurin summation gives
\[
\zeta(q)=\frac1{q-1}+\frac12+\frac q{12}
-\frac{q(q+1)(q+2)}{720}+\mathcal R(q).
\]
The fourth Bernoulli polynomial
\[
B_4(x)=x^4-2x^3+x^2-\frac1{30}
\]
satisfies $\sup_{0\le x\le1}|B_4(x)|=1/30$.
Consequently, the integral remainder satisfies
\[
\begin{aligned}
|\mathcal R(q)|
&\le
\frac{q(q+1)(q+2)(q+3)}{4!}\,
\frac1{30}\int_1^\infty t^{-q-4}\,\mathrm{d}t\\
&=\frac{q(q+1)(q+2)}{720}.
\end{aligned}
\]
Since $(q+1)(q+2)\le6$, it follows that
\[
\begin{aligned}
\zeta(q)
&\ge
\frac1{q-1}+\frac12+
q\left(\frac1{12}-\frac{(q+1)(q+2)}{360}\right)\\
&\ge
\frac1{q-1}+\frac12+\frac q{15}.
\end{aligned}
\]
Taking $q=1-a$ and weakening $1/15$ to $1/16$ to simplify
the algebra below, we obtain
\[
\zeta(1-a)\ge-\frac1a+\frac12+\frac{1-a}{16}.
\]
Next, the gamma product and reflection formulas yield
\[
p(a):=\frac{\sin(\pi a)}{\pi a(1-a)}
=\frac1{\Gamma(1+a)\Gamma(2-a)}
=\prod_{n=1}^{\infty}
\left(1+\frac{u}{n(n+1)}\right).
\]
Here the product identity uses
\[
1+\frac{u}{n(n+1)}
=\frac{(n+a)(n+1-a)}{n(n+1)}.
\]
Writing $q_n:=1/[n(n+1)]$, we have
\[
\sum_{n=1}^{\infty}q_n=1,
\qquad
\sum_{n=1}^{\infty}q_n^2
\le\frac12\sum_{n=1}^{\infty}q_n=\frac12.
\]
Hence
\[
\sum_{m<n}q_mq_n
=\frac12\left[
\left(\sum_{n=1}^{\infty}q_n\right)^2
-\sum_{n=1}^{\infty}q_n^2
\right]
\ge\frac14.
\]
Expanding the finite partial products, retaining the terms
of degrees zero, one, and two in $u$, and passing to the
limit gives
\[
p(a)\ge
1+u\sum_{n=1}^{\infty}q_n
+u^2\sum_{m<n}q_mq_n
\ge1+u+\frac{u^2}{4}.
\]
Since $0<u\le1/4$, we have $1-u/2>0$.
The preceding bounds therefore imply
\begin{align*}
u p(a)\left(\Psi(a)-\frac12\right)
&=p(a)\left(1-\frac u2\right)
+a(1+a)\zeta(1-a)\\
&\ge
\left(1+u+\frac{u^2}{4}\right)\left(1-\frac u2\right)
+a(1+a)\left(-\frac1a+\frac12+\frac{1-a}{16}\right)\\
&=\frac{u(1+a)}{16}
\bigl(1-2a(1-a)(2-a)\bigr)>0.
\end{align*}
The last inequality follows from
\[
2a(1-a)(2-a)\le\frac{2-a}{2}<1.
\]
As $u p(a)>0$, this proves $\Psi(a)>1/2$.

To prove $K<1/2$, use the expansion
\[
\log\Gamma(1/3)
=\log3-\frac\gamma3
+\sum_{n=2}^{\infty}
\frac{(-1)^n\zeta(n)}{n3^n}.
\]
The terms in this alternating series decrease strictly
in absolute value to zero, so
\[
\log\Gamma(1/3)
>\log3-\frac\gamma3+\frac{\zeta(2)}{18}
-\frac{\zeta(3)}{81}.
\]
The explicit formula for $C_{\rm conj}$ gives
\[
K=
3\log(2\pi)-\log3-6\log\Gamma(1/3)
+\gamma+1+e^{-1}.
\]
Consequently,
\[
K<
3\log(2\pi)-7\log3+3\gamma+1+e^{-1}
-\frac{\zeta(2)}3+\frac{2\zeta(3)}{27}
<\frac{106}{225}<\frac12.
\]
The second inequality follows from
\[
\log(2\pi)<1.838,\qquad
\log3>1.098,\qquad
\gamma<0.578,
\]
and
\[
e^{-1}<0.368,\qquad
\zeta(2)>1.644,\qquad
\zeta(3)<1.203,
\]
since substitution of these rational bounds gives
\[
3(1.838)-7(1.098)+3(0.578)+1+0.368
-\frac{1.644}{3}+\frac{2(1.203)}{27}
=\frac{106}{225}.
\]
Combining $\Psi(a)>1/2>K$ with \cref{eq:difconstants},
and using $\Gamma(a)>0$, proves
$C_{\rm TM}(s)>C_{\rm F}(s)$.
\end{proof}

\begin{proof}[Proof of \cref{cor:fekete-riesz-hypersingular}]
\leavevmode\phantomsection
\label{proof:fekete-riesz-hypersingular}
Use the notation from the proof of
\cref{thm:fekete-riesz-potentials}, so that
$a_d=\alpha/\sqrt d$ and
\[
N_i(t):=
\#\{j\ne i:\|\hat z_i-\hat z_j\|\le t\}.
\]
By \cref{eq:radial-count-remainder},
\[
N_i(t)\le\frac{dt^2}{4}+C(1+t\sqrt d),
\qquad 0\le t\le2,
\]
with a universal constant $C$.
For $t\ge a_d$, we have $t\sqrt d\ge\alpha$, and hence
\[
1+t\sqrt d
\le(\alpha^{-2}+\alpha^{-1})dt^2.
\]
Thus
\[
N_i(t)\le C'dt^2,
\qquad a_d\le t\le2,
\]
where $C':=1/4+C(\alpha^{-2}+\alpha^{-1})$ is universal.

Fix $s>2$. By \cref{eq:riesz-radial-representation},
\[
U_{s,i}(\omega_d)
\le(d-1)2^{-s}
+sC'd\int_{a_d}^2t^{1-s}\,\mathrm{d}t.
\]
Since
\[
\int_{a_d}^2t^{1-s}\,\mathrm{d}t
=\frac{a_d^{2-s}-2^{2-s}}{s-2}
\le\frac{\alpha^{2-s}}{s-2}d^{(s-2)/2},
\]
and $d\le d^{s/2}$, it follows that
\[
U_{s,i}(\omega_d)\le C_s d^{s/2},
\]
where $C_s$ depends only on $s$.

For the reverse bound, choose a universal $A\ge\alpha$
so large that
\[
\frac{A^2}{4}-C(1+A)\ge1.
\]
Fix any $d\ge2$ and any node $i$.
If $A/\sqrt d\le2$, then
\cref{eq:radial-count-remainder} gives
\[
N_i(A/\sqrt d)
\ge\frac{A^2}{4}-C(1+A)\ge1.
\]
If $A/\sqrt d>2$, every other node is at distance at most
$2<A/\sqrt d$ from $\hat z_i$.
In either case, there is a node $j\ne i$ such that
\[
\|\hat z_i-\hat z_j\|\le\frac A{\sqrt d}.
\]
Its contribution to the potential yields
\[
U_{s,i}(\omega_d)
\ge\|\hat z_i-\hat z_j\|^{-s}
\ge A^{-s}d^{s/2}.
\]
This proves the local comparison uniformly for all
$d\ge2$, all nodes, and all elliptic Fekete configurations.
Finally, summing over $i$ and using
\[
E_s(\omega_d)=\sum_{i=1}^d U_{s,i}(\omega_d)
\]
proves the total-energy comparison.
\end{proof}

\subsection{Discrete products and integrals}\label{sec:proofs:products}

\begin{proof}[Proof of \cref{th:fekete-discrete-Id-Ii}]
\leavevmode\phantomsection
\label{proof:fekete-discrete-Id-Ii}
\Cref{cor:fekete-separation,th:covering-radius-fekete} give universal constants $\alpha,\beta>0$ such that
\[
\delta(\omega_d)\ge\frac{\alpha}{\sqrt d},
\qquad
\rho(\omega_d)\le\frac{\beta}{\sqrt d}.
\]
Choose $0<\eta<\min\{b,\alpha/3\}$, where $b$ is given by \cref{lem:fekete-lagrange-local-stability}, and set
\[
B_j:=B_{\rm ch}\left(\hat{z}_j,\frac{\eta}{\sqrt d}\right),
\qquad
A_j:=B_{\rm ch}\left(\hat{z}_j,\frac{\eta}{\sqrt d}\right)
\setminus
B_{\rm ch}\left(\hat{z}_j,\frac{\eta}{2\sqrt d}\right).
\]
The separation estimate implies $B_j\subset V_j$. Moreover,
\[
\sigma(B_j)\asymp d^{-1},
\qquad
\sigma(A_j)\asymp d^{-1},
\]
and \cref{lem:fekete-lagrange-local-stability} yields
\[
H_{\omega_d\setminus\{\hat{z}_j\}}(p)\asymp P_j
\]
for $p\in B_j$, where the upper bound follows from the Fekete property.

On $V_j$ we have
\[
H_{\omega_d}(p)^2
=
\|p-\hat{z}_j\|^2H_{\omega_d\setminus\{\hat{z}_j\}}(p)^2.
\]
Since $H_{\omega_d\setminus\{\hat{z}_j\}}\le P_j$, $\|p-\hat{z}_j\|\le\rho(\omega_d)$, and $\sigma(V_j)\lesssim d^{-1}$,
\[
\int_{V_j}H_{\omega_d}(p)^2\,\mathrm{d}\sigma(p)
\lesssim
\frac{P_j^2}{d^2}.
\]
Conversely, on $A_j$ one has $H_{\omega_d\setminus\{\hat{z}_j\}}\gtrsim P_j$ and $\|p-\hat{z}_j\|\asymp d^{-1/2}$, so
\[
\int_{A_j}H_{\omega_d}(p)^2\,\mathrm{d}\sigma(p)
\gtrsim
\frac{P_j^2}{d^2}.
\]
The Voronoi cells cover $\Sphere^2$ and overlap only on sets
of $\sigma$-measure zero, whereas the annuli $A_j\subset V_j$
are pairwise disjoint. Summing the lower bounds over the
annuli and the upper bounds over the cells gives
\begin{align*}
\frac1{d^2}\sum_{j=1}^dP_j^2
&\lesssim
\sum_{j=1}^d\int_{A_j}H_{\omega_d}(p)^2\,\mathrm{d}\sigma(p)\\
&\le I_d
=\sum_{j=1}^d\int_{V_j}H_{\omega_d}(p)^2\,\mathrm{d}\sigma(p)\\
&\lesssim \frac1{d^2}\sum_{j=1}^dP_j^2.
\end{align*}
Fix now $i$. On $V_i$, the preceding local estimates give
\[
\int_{V_i}H_{\omega_d\setminus\{\hat{z}_i\}}(p)^2\,\mathrm{d}\sigma(p)
\asymp
\frac{P_i^2}{d}.
\]
If $j\ne i$, then on $V_j$,
\[
H_{\omega_d\setminus\{\hat{z}_i\}}(p)
=
H_{\omega_d\setminus\{\hat{z}_j\}}(p)\frac{\|p-\hat{z}_j\|}{\|p-\hat{z}_i\|}.
\]
Since $p\in V_j$,
\[
\|\hat{z}_i-\hat{z}_j\|
\le
\|p-\hat{z}_i\|+\|p-\hat{z}_j\|
\le
2\|p-\hat{z}_i\|.
\]
Together with $H_{\omega_d\setminus\{\hat{z}_j\}}\le P_j$, $\|p-\hat{z}_j\|\lesssim d^{-1/2}$, and $\sigma(V_j)\lesssim d^{-1}$, this gives
\[
\int_{V_j}H_{\omega_d\setminus\{\hat{z}_i\}}(p)^2\,\mathrm{d}\sigma(p)
\lesssim
\frac{P_j^2}{d^2\|\hat{z}_i-\hat{z}_j\|^2}.
\]
For the reverse inequality, restrict the integral to $A_j$. 
There
\[
H_{\omega_d\setminus\{\hat{z}_j\}}(p)\gtrsim P_j,
\qquad
\|p-\hat{z}_j\|\asymp d^{-1/2},
\]
and, by separation,
\[
\|p-\hat{z}_i\|
\le
\|\hat{z}_i-\hat{z}_j\|+\frac{\eta}{\sqrt d}
\lesssim \|\hat{z}_i-\hat{z}_j\|.
\]
Hence
\[
\int_{A_j}H_{\omega_d\setminus\{\hat{z}_i\}}(p)^2\,\mathrm{d}\sigma(p)
\gtrsim
\frac{P_j^2}{d^2\|\hat{z}_i-\hat{z}_j\|^2}.
\]
Summing over all Voronoi cells proves
\[
I_{d,i}
\asymp
\frac{P_i^2}{d}
+
\frac1{d^2}\sum_{j\ne i}\frac{P_j^2}{\|\hat{z}_i-\hat{z}_j\|^2}.
\]
\end{proof}

\begin{proof}[Proof of \cref{cor:average-Ii-Id-log-gap}]
\leavevmode\phantomsection
\label{proof:average-Ii-Id-log-gap}
By \cref{th:fekete-discrete-Id-Ii}, interchanging the finite
sums gives
\[
\sum_{i=1}^d I_{d,i}
\asymp
\frac1d\sum_{j=1}^d P_j^2
+\frac1{d^2}\sum_{j=1}^d P_j^2
\sum_{\substack{1\le i\le d\\i\ne j}}
\frac1{\|\hat z_i-\hat z_j\|^2}.
\]
By \cref{eq:fekete-local-riesz-two}, the inner sums are
comparable to $d\log d$, uniformly in $j$ and the configuration.
For the finitely many small degrees, the lower comparison
follows from $\|\hat z_i-\hat z_j\|\le2$, and the upper one
from separation. Therefore,
\[
\sum_{i=1}^d I_{d,i}
\asymp
\frac{\log d}{d}\sum_{j=1}^d P_j^2
\asymp d\log d\,I_d.
\]
A bound $I_{d,i}\le C I_d$ uniform in $i,d$ would instead give
$\sum_{i=1}^d I_{d,i}\le CdI_d$, contradicting this comparison
as $d\to\infty$.
\end{proof}

\begin{proof}[Proof of \cref{cor:fekete-norm-comparisons}]
\leavevmode\phantomsection
\label{proof:fekete-norm-comparisons}
By \cref{lem:fekete-node-is-max}, the deleted product has supremum $P_i$, hence $I_{d,i}^{1/2}\le P_i$. 
Applying \cref{lem:spherical-nikolskii} to that degree-$(d-1)$ product gives $P_i\le\sqrt d\,I_{d,i}^{1/2}$. Also \cref{th:fekete-discrete-Id-Ii} yields
\[
I_d\lesssim\frac1{d^2}\sum_{j=1}^dP_j^2\le\frac1d\max_jP_j^2.
\]
Finally, substitute the upper bound for $P_i$ into \cref{eq:norm_cond} to get the local conditioning inequality.
\end{proof}

\begin{lemma}\label{lem:maximal-vandermonde-growth}
For every $d\ge2$,
\[
M_{d+1}^{d-1}\le M_d^{d+1}.
\]
If $\omega_d$ is elliptic Fekete, then
\[
\mathcal Q(\omega_d)\le\left(\frac{M_{d+1}}{M_d}\right)^{1/2}
\le M_d^{1/(d-1)}.
\]
\end{lemma}
\begin{proof}
\leavevmode\phantomsection
\label{proof:maximal-vandermonde-growth}
Start with a $(d+1)$-point Fekete configuration, and delete each of its points in turn. 
Let $M(\omega_d^{(m)})$ be the ordered product for the configuration obtained by the $m$-th deletion. 
Each ordered pair of distinct original nodes survives in exactly $d-1$ of the deleted configurations. 
Therefore
\[
M_{d+1}^{d-1}=\prod_{m=1}^{d+1}M(\omega_d^{(m)})\le M_d^{d+1}.
\]
Next, adjoining any point $p$ to a $d$-point Fekete configuration gives
\[
M_dH_{\omega_d}(p)^2\le M_{d+1}.
\]
Taking the maximum proves the first bound for $\mathcal Q$. The preceding growth inequality implies $M_{d+1}/M_d\le M_d^{2/(d-1)}$, which proves the second.
\end{proof}

We prove the maximum-product comparison at this point because it also supplies the individual upper bounds at the exponential scale.

\begin{proof}[Proof of \cref{prop:fekete-polarization-max-product}]
\leavevmode\phantomsection
\label{proof:fekete-polarization-max-product}
For $p\in\Sphere^2$, choose a nearest node $\hat{z}_i$. By \cref{lem:fekete-node-is-max,th:covering-radius-fekete},
\[
H_{\omega_d}(p)=\|p-\hat{z}_i\|H_{\omega_d\setminus\{\hat{z}_i\}}(p)
\le\rho(\omega_d)P_i\lesssim\frac1{\sqrt d}\max_jP_j.
\]
For the reverse inequality, choose $i$ with $P_i=\max_jP_j$ and choose $p$ at chordal distance $b/(2\sqrt d)$ from $\hat{z}_i$, reducing the universal $b$ in \cref{lem:fekete-lagrange-local-stability} to at most $1$ if necessary. That lemma gives
\[
H_{\omega_d}(p)\ge\frac{ab}{2\sqrt d}\,P_i.
\]
Taking the maximum over $p$ completes the comparison.
\end{proof}

\begin{proof}[Proof of \cref{th:fekete-exponential-scale-integrals}]
\leavevmode\phantomsection
\label{proof:fekete-exponential-scale-integrals}
First, \cref{eq:asymp_min_energy} gives
\begin{align*}
\frac{\log M_d}{d-1}
&=\frac{-\kappa d^2+\frac12d\log d-C_{\log}d+o(d)}{d-1}\\
&=-\kappa d+\frac12\log d-\kappa-C_{\log}+o(1).
\end{align*}
Exponentiating proves 
\begin{equation}\label{eq:Md-root-asymptotic}
M_d^{1/(d-1)}
=\sqrt d\,e^{-\kappa d-\kappa-C_{\log}+o(1)},
\end{equation}
and in particular $M_d^{1/(d-1)}\lesssim\sqrt d\,e^{-\kappa d}$ for all $d\ge2$, after enlarging the constant for finitely many degrees. Combining \cref{lem:maximal-vandermonde-growth,prop:fekete-polarization-max-product} yields
\[
P_i\le\max_jP_j\lesssim\sqrt d\,\mathcal Q(\omega_d)
\lesssim d\,e^{-\kappa d}.
\]
The lower bound for $P_i$ follows from \cref{lem:fekete-node-is-max,lem:product-jensen} applied to the deleted product.

By \cref{th:fekete-discrete-Id-Ii} and this upper bound,
\[
I_d\lesssim\frac1{d^2}\sum_{j=1}^dP_j^2\lesssim d\,e^{-2\kappa d}.
\]
Similarly, using \cref{eq:fekete-local-riesz-two},
\begin{align*}
I_{d,i}
&\lesssim\frac{P_i^2}{d}
+\frac1{d^2}\sum_{j\ne i}\frac{P_j^2}{\|\hat{z}_i-\hat{z}_j\|^2}\\
&\lesssim d\,e^{-2\kappa d}
+e^{-2\kappa d}U_{2,i}(\omega_d)
\lesssim d\log d\,e^{-2\kappa d}.
\end{align*}
The lower bounds for $I_d^{1/2}$ and $I_{d,i}^{1/2}$ are \cref{lem:product-jensen}. Taking square roots proves all the assertions.
\end{proof}

\begin{remark}\label{rem:deletion-versus-adjoining}
For $d\ge2$, deleting a node from a Fekete configuration gives $M_d/P_i^2\le M_{d-1}$, which bounds $P_i$ from below. 
The upper bound uses the complementary operation of adjoining a point and the growth inequality in \cref{lem:maximal-vandermonde-growth}. 
\end{remark}

\subsection{Logarithmic polarization}\label{sec:proofs:polarization}

\begin{proof}[Proof of \cref{th:fekete-wagner-polarization}]
\leavevmode\phantomsection
\label{proof:fekete-wagner-polarization}
By the definition of $\mathcal Q_d$ and
\cref{lem:maximal-vandermonde-growth},
\begin{equation*}
\mathcal Q_d\le\mathcal Q(\omega_d)
\le\left(\frac{M_{d+1}}{M_d}\right)^{1/2}
\le M_d^{1/(d-1)}.
\end{equation*}
By \cref{eq:Md-root-asymptotic} we have $\mathcal Q(\omega_d)\lesssim\sqrt d\,e^{-\kappa d}$
for all $d\ge2$, after adjusting the implicit constant for
finitely many degrees. On the other hand,
\cref{lem:product-jensen} gives $\mathcal Q_d\ge e^{-\kappa d}$.
Therefore,
\[
0\le
\log\frac{\mathcal Q(\omega_d)}{\mathcal Q_d}
\le\frac12\log d+O(1).
\]
\end{proof}

\begin{proof}[Proof of \cref{prop:fekete-polarization-upper-balance}]
\leavevmode\phantomsection
\label{proof:fekete-polarization-upper-balance}
If $\mathcal Q(\omega_d)\asymp e^{-\kappa d}$, \cref{prop:fekete-polarization-max-product} gives
\[
P_i\le\max_jP_j\lesssim\sqrt d\,\mathcal Q(\omega_d)
\lesssim\sqrt d\,e^{-\kappa d}.
\]
Conversely, that uniform upper bound and the same proposition yield $\mathcal Q(\omega_d)\lesssim e^{-\kappa d}$. The reverse inequality $\mathcal Q(\omega_d)\ge e^{-\kappa d}$ follows from \cref{lem:product-jensen} for every configuration.
\end{proof}

\subsection{Conditioning and log-balancedness}\label{sec:proofs:conditioning}

\begin{proof}[Proof of \cref{prop:optcond_Id}]
\leavevmode\phantomsection
\label{proof:optcond_Id}
The lower bound is \cref{lem:product-jensen}. For the upper bound, take geometric means in \cref{eq:norm_cond}. Since $\prod_iP_i\le M_d$ for every configuration,
\[
\frac{\sqrt{d(d+1)}}2\frac{I_d^{1/2}}{M_d^{1/d}}
\le\left(\prod_{i=1}^d\cond{f_d}{z_i}\right)^{1/d}
\le\condtotal{f_d}\le C\sqrt d.
\]
Consequently
\[
I_d^{1/2}\le\frac{2C}{\sqrt{d+1}}M_d^{1/d}
\lesssim_C e^{-\kappa d}
\]
by \cref{eq:fekete-geometric-mean}. Although that formula identifies the geometric mean for a minimizer, $M_d$ itself depends only on $d$; no minimality is assumed for the roots of $f_d$.
\end{proof}

The following elementary disk estimate is the analytic ingredient for \cref{th:optimal-conditioning-riesz}. It controls the number and the inverse-square interactions of selected zeros; other zeros of the function need not be excluded.

\begin{lemma}\label{lem:holomorphic-close-zeros}
Let $\Omega\subset\C$ be an open set containing the closed
disk $\{w\in\C:|w|\le4\}$, and let $g:\Omega\to\C$ be
holomorphic. Assume that, for some $M>0$,
\[
|g(w)|\le M,\qquad |w|\le4.
\]
Let $a_1,\ldots,a_k$ be distinct zeros of $g$ in
$\{w\in\C:|w|\le1\}$, and suppose that
\[
|g'(a_i)|\ge\eta,\qquad 1\le i\le k,
\]
for some $\eta>0$. Then there is an absolute constant $C>0$
such that
\[
k^2+\sum_{\substack{1\le i,j\le k\\i\ne j}}
|a_i-a_j|^{-2}
\le C\frac{M^2}{\eta^2}.
\]
For $k=0$, both terms on the left are understood to be zero.
\end{lemma}

\begin{proof}
\leavevmode\phantomsection
\label{proof:holomorphic-close-zeros}
For $|a|\le1$, set
\[
b_a(w):=\frac{4(w-a)}{16-\overline a w}.
\]
For $|a|\le1$ and $|w|\le4$,
\[
|16-\overline a\,w|
\ge 16-|a||w|
\ge 12,
\]
so $b_a$ is holomorphic on a neighborhood of the closed disk
and has there a unique zero, namely the simple zero $w=a$.
For $|w|=4$,
\[
16-\overline a\,w
=
w\,\overline{(w-a)},
\]
and hence $|b_a(w)|=1$.
Since $g(a_j)=0$, the factor $w-a_j$ cancels in the quotient
\[
q(w):=\frac{g(w)}{\prod_{j=1}^k b_{a_j}(w)}.
\]
Thus $q$ extends holomorphically across each $a_j$.
On $|w|=4$, we have $|q(w)|=|g(w)|\le M$, so the maximum
modulus principle gives $|q(w)|\le M$ throughout the disk.
Also,
\[
|b_a'(a)|=\frac4{16-|a|^2}\le\frac4{15},
\qquad
|b_{a_j}(a_i)|\le\frac4{15}|a_i-a_j|\le\frac8{15}=:\theta<1.
\]
For $k\ge1$, differentiation at a selected zero gives
\[
\eta\le\frac4{15}M\theta^{k-1}.
\]
Squaring the preceding inequality and multiplying by $k^2$
gives
\[
\eta^2k^2
\le\frac{16}{225}M^2 k^2\theta^{2(k-1)}
\le CM^2.
\]
Here $C$ is absolute, since $\theta=8/15$ is fixed and
$k^2\theta^{2(k-1)}\to0$ as $k\to\infty$.
Dividing by $\eta^2$ proves $k^2\le CM^2/\eta^2$.
If $k\ge2$, retaining the distance to one particular other zero gives
\[
\eta\le\left(\frac4{15}\right)^2M|a_i-a_j|\theta^{k-2},
\qquad i\ne j.
\]
Squaring the reciprocal inequality and summing yields
\[
\sum_{i\ne j}|a_i-a_j|^{-2}
\le C\frac{M^2}{\eta^2}k(k-1)\theta^{2(k-2)}
\le C'\frac{M^2}{\eta^2}.
\]
For $k=1$ this sum is empty. 
This proves the lemma.
\end{proof}

\begin{lemma}\label{lem:conditioning-close-pairs}
Let $f$ have degree $d\ge2$, simple roots, and $\condtotal{f}\le C_0\sqrt d$. 
For its spherical root configuration, set
\[
A_d(t):=
\#\{(i,j):i\ne j,\ \|\hat{z}_i-\hat{z}_j\|\le t\}.
\]
There is a constant $C_{C_0}$, independent of $d$ and $f$, such that
\begin{equation}\label{eq:conditioning-microscopic-pairs}
\sum_{\substack{i\ne j\\
\|\hat{z}_i-\hat{z}_j\|\le d^{-1/2}}}
\frac1{\|\hat{z}_i-\hat{z}_j\|^2}
\le C_{C_0}d^2
\end{equation}
and
\begin{equation}\label{eq:conditioning-pair-count}
A_d(t)\le C_{C_0}d^2t^2,\qquad 0<t\le2.
\end{equation}
Consequently, the tail $T_d$ defined in \cref{eq:transfer-tail-hypothesis} satisfies $T_d\le C_{C_0}d^2$.
\end{lemma}

\begin{proof}
\leavevmode\phantomsection
\label{proof:conditioning-close-pairs}
Throughout this proof, write
$r_{ij}:=\|\hat z_i-\hat z_j\|$.
View the homogeneous form associated with $f$ as a section $F$ of $\mathcal O(d)$, normalized by
\[
\int_{\Sphere^2}|F|_{d\phi_{\rm FS}}^2\,\mathrm{d}\sigma=1.
\]
By \cref{eq:bombieri-spherical-norm}, $\|F\|_{\mathrm{BW}}=\sqrt{d+1}$.
This normalization does not change condition numbers.

For any center $q\in\Sphere^2$, use a unitary coordinate change taking $q$ to $[1:0]$, and represent $F$ by a polynomial $G_q(z)$ in the affine coordinate $z=v/u$. Put $h=d^{-1/2}$ and $g_q(w)=G_q(hw)$. 
Let $\Omega_q$ be the spherical image of $|w|<5$, and write
\[
m_q:=\int_{\Omega_q}|F|_{d\phi_{\rm FS}}^2\,\mathrm{d}\sigma
=\frac1{\pi d}\int_{|w|<5}
\frac{|g_q(w)|^2}{(1+|w|^2/d)^{d+2}}\,\mathrm{d}A(w).
\]
Here $\mathrm{d}A$ is planar area. 
The weight lies between $e^{-50}$ and $1$ on this fixed disk for all $d\ge2$, since $(d+2)\log(1+25/d)\le50$. 
The submean inequality on disks of radius $1/2$ centered at $|w|\le4$ therefore gives
\begin{equation}\label{eq:conditioning-chart-mass}
\max_{|w|\le4}|g_q(w)|^2
\le C\int_{|w|<5}|g_q(w)|^2\,\mathrm{d}A(w)
\le C'd\,m_q.
\end{equation}
At a zero $a$ with $|a|\le1$, the projective version of \cref{eq:condition-bombieri}, valid in every unitary chart, gives
\[
|g_q'(a)|=
\frac{\sqrt{d+1}}{\mu(F,[1:ha])}
(1+|a|^2/d)^{(d-2)/2}\ge\frac1{C_0}.
\]
A possible zero at infinity in this chart plays no role. 
Denote the zeros in $|w|\le1$ by $a_1,\ldots,a_{k_q}$. 
Thus \cref{lem:holomorphic-close-zeros,eq:conditioning-chart-mass} imply
\begin{equation}\label{eq:conditioning-chart-zero-bound}
k_q^2+\sum_{i\ne j}|a_i-a_j|^{-2}\le C_{C_0}d\,m_q.
\end{equation}
For these zeros, the chordal formula gives
\[
r_{ij}=\frac{2h|a_i-a_j|}
{\sqrt{1+|a_i|^2/d}\sqrt{1+|a_j|^2/d}},
\qquad r_{ij}^{-2}\le Cd\,|a_i-a_j|^{-2}.
\]
Choose a finite set of auxiliary centers $q_\alpha\in\Sphere^2$
such that
\[
\|q_\alpha-q_\beta\|\ge h/4,\qquad \alpha\ne\beta,
\]
and no further point of $\Sphere^2$ can be added while
preserving this separation.
Such a set exists by compactness.
The caps $B_{\rm ch}(q_\alpha,h/4)$ cover $\Sphere^2$:
otherwise, any point outside their union could be added
to the set of centers.
The caps $B_{\rm ch}(q_\alpha,10h)$ have bounded overlap.
Indeed, the caps $B_{\rm ch}(q_\alpha,h/8)$ are pairwise
disjoint. If a center $q_\alpha$ satisfies
$p\in B_{\rm ch}(q_\alpha,10h)$, then its $h/8$ cap is
contained in a cap centered at $p$ of radius $10h+h/8$.
When this radius is at most $2$, the chordal cap-area formula
bounds the number of such centers by an absolute constant.
In the remaining case $h$ is bounded below by an absolute
constant, so $d$ ranges over only finitely many values and the
same conclusion follows, after increasing the constant.
Since $\Omega_{q_\alpha}$ is contained in the $10h$ cap,
\[
\sum_\alpha m_{q_\alpha}\le C.
\]
The image of $|w|\le1$ is the chordal cap of radius $2/\sqrt{d+1}$ about $q_\alpha$. 
Every pair with $r_{ij}\le h$ belongs to at least one such cap: choose a center within $h/4$ of its first point; the second is within $5h/4$, and $5h/4<2/\sqrt{d+1}$ for $d\ge2$. Sum \cref{eq:conditioning-chart-zero-bound} over the centers and use the chordal comparison. 
All summands are nonnegative, so multiple counting is harmless. 
This proves \cref{eq:conditioning-microscopic-pairs}.

In all inverse-square sums below the diagonal is omitted. 
For the aggregate counting bound, assign each point of
$\Sphere^2$ to a nearest center $q_\alpha$; if several centers
are equally close, choose the one with the smallest index.
Let $W_\alpha$ be the set assigned to $q_\alpha$ and put
$k_\alpha:=\#(\omega_d\cap W_\alpha)$.
The sets $W_\alpha$ form a disjoint partition of $\Sphere^2$.
Each cell lies in the $h/4$ cap and hence in its chart's unit disk. 
By \cref{eq:conditioning-chart-zero-bound},
\[
\sum_\alpha k_\alpha^2\le C_{C_0}d.
\]
If $t\ge h$, a pair of points in $W_\alpha,W_\beta$ at distance at most $t$ requires $\|q_\alpha-q_\beta\|\le t+h/2$. Packing the separated centers shows that for each $\alpha$ there are at most $Cdt^2$ such $\beta$. This relation is symmetric; using $2k_\alpha k_\beta\le k_\alpha^2+k_\beta^2$ gives
\[
A_d(t)\le\sum_{\|q_\alpha-q_\beta\|\le t+h/2}k_\alpha k_\beta
\le Cdt^2\sum_\alpha k_\alpha^2\le C_{C_0}d^2t^2.
\]
For $t<h$, \cref{eq:conditioning-microscopic-pairs} instead gives
\[
A_d(t)\le t^2\sum_{r_{ij}\le h}r_{ij}^{-2}\le C_{C_0}d^2t^2.
\]
This proves \cref{eq:conditioning-pair-count}. It is an aggregate estimate over pairs, not a count uniform about every center.

Finally, the pairs with $r_{ij}\le h$ contribute at most $C_{C_0}d^2$ to $T_d$. 
For the annulus $2^m h<r_{ij}\le2^{m+1}h$, \cref{eq:conditioning-pair-count} and $x_{ij}^{d+1}\le e^{-dr_{ij}^2/4}$ bound its contribution by $C_{C_0}d^2e^{-4^m/4}$. 
If its outer radius exceeds $2$, use $A_d(2)\le d(d-1)$ instead. 
Summing over $m\ge0$ proves the stated bound on $T_d$.
\end{proof}

\begin{proof}[Proof of \cref{th:optimal-conditioning-riesz}]
\leavevmode\phantomsection
\label{proof:optimal-conditioning-riesz}
By \cref{eq:norm_cond,lem:product-jensen} and $\condtotal{f_d}\le C_0\sqrt d$,
\[
P_i\ge\frac{\sqrt{d+1}}{2C_0}\,I_d^{1/2}
\ge\frac{\sqrt{d+1}}{2C_0}\,e^{-\kappa d}.
\]
Taking logarithms and summing proves
\[
\Energy(\omega_d)\le\kappa d^2-\frac d2\log(d+1)+d\log(2C_0)
\le\kappa d^2-\frac12d\log d+A_{C_0}d
\]
with $A_{C_0}=\log(2C_0)\in\R$. This is the bound in \cite[equation~(9)]{Etayo2021}, reproduced here to fix all normalizations. \Cref{lem:conditioning-close-pairs} gives the tail hypothesis with $B$ depending only on $C_0$. Applying \cref{lem:log-to-riesz-transfer} proves both assertions.
\end{proof}

\begin{proof}[Proof of \cref{th:local-condition-implies-log-balanced}]
\leavevmode\phantomsection
\label{proof:local-condition-implies-log-balanced}
The upper local bound implies $\condtotal{f_d}\le C\sqrt d$. Hence, by \cref{prop:optcond_Id},
\[
I_d^{1/2}=\exp(-\kappa d+O(1)).
\]
Using \cref{eq:norm_cond},
\[
P_i
=
\frac{\sqrt{d(d+1)}}{2}\,
\frac{I_d^{1/2}}{\cond{f_d}{z_i}}.
\]
The two-sided local bounds therefore give
\[
P_i=\sqrt d\,\exp(-\kappa d+O(1))
\]
uniformly in $d$ and $1\le i\le d$. Thus $\{\omega_d\}_{d\in\mathcal D}$ is log-balanced.
\end{proof}

\begin{proof}[Proof of \cref{cor:log-balanced-local-lower-bound}]
\leavevmode\phantomsection
\label{proof:log-balanced-local-lower-bound}
By log-balancedness,
\[
P_i\asymp \sqrt d\,e^{-\kappa d}
\]
uniformly in $d$ and $i$. Substitution in \cref{eq:norm_cond} gives
\[
\cond{f_d}{z_i}
\asymp
\sqrt d\,e^{\kappa d}I_d^{1/2}.
\]
\Cref{lem:product-jensen} gives $I_d^{1/2}\ge e^{-\kappa d}$, and the lower bound follows.
\end{proof}

\begin{proof}[Proof of \cref{prop:condition-weights}]
\leavevmode\phantomsection
\label{proof:condition-weights}
For every polynomial with simple roots, \cref{eq:norm_cond} gives
\[
\mu_i^{-2}=\frac{4P_i^2}{d(d+1)I_d},
\qquad
\sum_{j=1}^d\mu_j^{-2}
=\frac{4\sum_{j=1}^dP_j^2}{d(d+1)I_d}.
\]
Dividing proves \cref{eq:condition-weight-identity}.
If the configuration is Fekete,
\cref{th:fekete-discrete-Id-Ii} gives
\[
I_d\asymp\frac1{d^2}\sum_{j=1}^dP_j^2,
\]
which proves \cref{eq:fekete-condition-weights}.
Taking maxima of $\mu_i^2\asymp1/w_i$ proves
\cref{eq:global-condition-weights}.
\end{proof}

\begin{proof}[Proof of \cref{prop:fekete-log-balanced-deletion}]
\leavevmode\phantomsection
\label{proof:fekete-log-balanced-deletion}
For $d\ge2$, \cref{eq:fekete-condition-weights} identifies
$w_i\asymp d^{-1}$ with $\cond{f_d}{z_i}\asymp\sqrt d$.
Log-balancedness gives $\sum_jP_j^2\asymp d^2e^{-2\kappa d}$
and hence $w_i\asymp d^{-1}$. Conversely, two-sided bounds
$w_i\asymp d^{-1}$ imply $P_i/P_j\asymp1$ uniformly.
Taking geometric means over $j$ and using
\cref{eq:fekete-geometric-mean} then gives
\[
P_i\asymp M_d^{1/d}\asymp\sqrt d\,e^{-\kappa d}.
\]
This proves \cref{eq:balanced-condition-weights}.

Since the energy counts ordered pairs, deleting $\hat z_i$ gives
$\Energy(\omega_d\setminus\{\hat z_i\})=\Energy(\omega_d)+2\log P_i$.
Consequently,
\begin{equation*}
\Energy\!\left(\omega_d\setminus\{\hat z_i\}\right)
-\Energy\!\left(\omega_d\setminus\{\hat z_j\}\right)
=2\log\!\left(\frac{P_i}{P_j}\right).
\end{equation*}
A uniform bound on these differences is equivalent to
$P_i/P_j\asymp1$, which is equivalent to log-balancedness by
the same geometric-mean argument. The degree $d=1$, if present,
is absorbed into the constants, with $P_1=1$ and
$\Energy(\varnothing)=0$. The deletion identity holds for any
configuration; minimality is used to fix the geometric mean.
\end{proof}

\begin{proof}[Proof of \cref{prop:fekete-oscillation-local-condition}]
\leavevmode\phantomsection
\label{proof:fekete-oscillation-local-condition}
Since every chordal distance is at most $2$,
\cref{lem:fekete-pairwise-products} gives
\[
\frac{\max_i P_i}{\min_i P_i}\lesssim\sqrt d.
\]
For the local condition number, fix $i$. Squaring the lower
bound in \cref{eq:fekete-pairwise-products} gives
\[
\left(\frac{P_j}{P_i}\right)^2
\gtrsim
\frac1{d\,\|\hat z_i-\hat z_j\|^2},
\qquad j\ne i.
\]
Summing over the other nodes, we obtain
\[
\sum_{j=1}^d\left(\frac{P_j}{P_i}\right)^2
=
1+\sum_{j\ne i}\left(\frac{P_j}{P_i}\right)^2
\gtrsim
\frac1d\sum_{j\ne i}\frac1{\|\hat z_i-\hat z_j\|^2}
=\frac1d U_{2,i}(\omega_d).
\]
By \cref{eq:fekete-local-riesz-two},
\[
\frac1d U_{2,i}(\omega_d)\gtrsim\log(d+1)
\]
for all sufficiently large degrees, uniformly in the node
and configuration. For the finitely many smaller degrees,
the bound $U_{2,i}(\omega_d)\ge(d-1)/4$ gives the same
comparison after decreasing the implicit constant.
Using \cref{eq:fekete-condition-weights}, we conclude that
\[
\cond{f_d}{z_i}^2
\asymp\frac1{w_i}
=\sum_{j=1}^d\left(\frac{P_j}{P_i}\right)^2
\gtrsim\log(d+1).
\]
Taking square roots proves \cref{eq:fekete-local-condition-log}.
The same estimate gives $w_i\lesssim1/\log(d+1)$.
Finally, taking logarithms of the product-ratio bound yields
\[
\max_i\log P_i-\min_i\log P_i
\le\frac12\log d+O(1).
\]
\end{proof}

\begin{proof}[Proof of \cref{cor:log-balanced+Fekete-local-lower-bound}]
\leavevmode\phantomsection
\label{proof:log-balanced+Fekete-local-lower-bound}
Log-balancedness gives $P_j^2\asymp d\,e^{-2\kappa d}$ uniformly. Substitution into \cref{th:fekete-discrete-Id-Ii} gives
\[
I_d\asymp\frac1{d^2}\sum_{j=1}^d d\,e^{-2\kappa d}=e^{-2\kappa d}
\]
and
\[
I_{d,i}\asymp e^{-2\kappa d}\left(1+\frac1d U_{2,i}(\omega_d)\right).
\]
By \cref{eq:fekete-local-riesz-two}, the last parenthesis is comparable to $\log d$, uniformly in $i$; finite small degrees are absorbed into the constants. Hence $I_{d,i}\asymp(\log d)e^{-2\kappa d}$ and $I_{d,i}/I_d\asymp\log d$. Taking square roots gives the integral estimates, and \cref{eq:norm_cond} gives $\cond{f_d}{z_i}\asymp\sqrt d$. The constants depend at most on the log-balancedness bounds of the family and the universal Fekete constants.
\end{proof}

\begin{remark}\label{rem:log-balanced-lagrange-norm}
For a log-balanced Fekete family, the same calculation gives
\[
\frac{I_{d,i}}{P_i^2}\asymp\frac{\log d}{d}.
\]
The comparison constants depend only on the log-balancedness
bounds of the family. Thus even in the log-balanced regime the one-point-deleted norm has a logarithmic contribution that is not captured by its mass on a single microscopic ball.
\end{remark}

\begin{proof}[Proof of \cref{thm:optimal-condition-not-log-balanced}]
\leavevmode\phantomsection
\label{proof:optimal-condition-not-log-balanced}
For $M\ge 5$, set $d=4M^2$ and define
\[
h_j=1-\frac{j^2}{M^2},
\qquad
\varrho(t)=\sqrt{\frac{1+t}{1-t}}.
\]
Consider the polynomial studied in \cite[Theorem~1.3]{BeltranLizarte2022}, whose spherical roots form a variant of the Diamond ensemble introduced by Beltr\'an and Etayo \cite{BeltranEtayo2020}:
\[
f_d(z)
=
\bigl(z^{4M}-1\bigr)
\prod_{j=1}^{M-1}
\bigl(z^{4j}-\varrho(h_j)^{4j}\bigr)
\bigl(z^{4j}-\varrho(h_j)^{-4j}\bigr).
\]
The degree is $4M+8\sum_{j=1}^{M-1}j=4M^2=d$. The radii of its distinct factors are different and each factor has simple roots, so all roots of $f_d$ are simple. 
By \cite[Theorem~1.3]{BeltranLizarte2022},
\[
\condtotal{f_d}\le \frac{19}{2}\sqrt{d+1}.
\]
Under inverse stereographic projection, the roots of \(f_d\) form the
configuration considered in \cite{BeltranLizarte2022}. Let
$\hat{z}_{i_M}$ be any point on its first northern parallel.
To make the two estimates used from that paper explicit, for $h\in(-1,1)$ and $q\in\Sphere^2$ set
\[
\mathcal A_h(q):=\frac1{2\pi}\int_0^{2\pi}
\log\left\|q-\bigl(\sqrt{1-h^2}\cos\theta,\sqrt{1-h^2}\sin\theta,h\bigr)\right\|\,\mathrm{d}\theta
\]
and
\[
S_d(q):=4M\mathcal A_0(q)+\sum_{j=1}^{M-1}4j\bigl(\mathcal A_{h_j}(q)+\mathcal A_{-h_j}(q)\bigr).
\]
The logarithmic singularity in a parallel integral is integrable. 
Thus $S_d(q)$ is obtained by replacing the sum of logarithmic distances from $q$ to the nodes on each parallel by the corresponding angular average multiplied by the number of nodes.
For $q=\hat{z}_{i_M}$ on the first northern parallel, \cite[Proposition~4.3]{BeltranLizarte2022}, with $\ell=1$, gives
\[
S_d(q)+\kappa d\ge-1+\frac13\log\frac{M+1}{2}.
\]
On the other hand, \cite[Proposition~6.3]{BeltranLizarte2022} gives
\[
\log P_{i_M}\ge S_d(q)+\log(2\sqrt2\,M)-\frac18.
\]
Adding these two bounds yields
\[
\log P_{i_M}
\ge
-\kappa d
+\log(2\sqrt{2}\,M)
+\frac{1}{3}\log\left(\frac{M+1}{2}\right)
-\frac{9}{8}.
\]
Consequently,
\[
P_{i_M}
\ge
2\sqrt{2}\,M\,e^{-\kappa d-9/8}
\left(\frac{M+1}{2}\right)^{1/3}.
\]
Since \(\sqrt d=2M\), it follows that
\[
\frac{P_{i_M}}{\sqrt d\,e^{-\kappa d}}
\ge
\sqrt{2}\,e^{-9/8}
\left(\frac{M+1}{2}\right)^{1/3}
\ge
2^{-1/6}e^{-9/8}d^{1/6}.
\]
The right-hand side tends to infinity as $M\to\infty$. Hence the uniform upper bound
\[
P_i\lesssim \sqrt d\,e^{-\kappa d},
\]
required by log-balancedness, fails along this sequence.

Finally, \cref{prop:optcond_Id} gives $I_d^{1/2}\lesssim e^{-\kappa d}$. Therefore, by \cref{eq:norm_cond},
\[
\cond{f_d}{z_{i_M}}
\lesssim
d\,\frac{e^{-\kappa d}}{d^{2/3}e^{-\kappa d}}
=
d^{1/3}.
\]
This completes the proof.
\end{proof}

\subsection*{AI Declaration}

The author used ChatGPT (OpenAI) as an interactive assistant during the development and preparation of this work. 
Its use included exploring mathematical questions, formulating candidate results, developing and revising proof arguments, checking calculations and consistency, identifying relevant literature, and editing the exposition and LaTeX source. 
The author directed the investigation, selected and revised the material included in the manuscript, and takes full responsibility for its mathematical claims, proofs, bibliographic attributions, and final presentation.

\bibliographystyle{amsplain}
\bibliography{Bib}

\end{document}